\documentclass{amsart}
\usepackage{mathrsfs}
\usepackage{mathrsfs}
\usepackage{mathrsfs}
\usepackage{mathrsfs}
\usepackage{amsfonts}
\usepackage{cases}
\usepackage{latexsym}
\usepackage{amsmath}
\usepackage[arrow,matrix]{xy}
\usepackage{stmaryrd}
\usepackage{amsfonts}
\usepackage{amsmath,amssymb,amscd,bbm,amsthm,mathrsfs,dsfont}
\usepackage{fancyhdr}
\usepackage{amsxtra,ifthen}
\usepackage{verbatim}

\numberwithin{equation}{section}

\theoremstyle{plain}
\newtheorem{theorem}{Theorem}[section]
\newtheorem{lemma}[theorem]{Lemma}
\newtheorem{proposition}[theorem]{Proposition}
\newtheorem{corollary}[theorem]{Corollary}
\newtheorem{conjecture}[theorem]{Conjecture}

\theoremstyle{definition}

\newtheorem{question}[theorem]{Question}

\begin{document}

\title[Centralizing Traces and Lie Triple Isomorphisms]
{Centralizing Traces and Lie Triple Isomorphisms on Triangular
Algebras}

\author{Xinfeng Liang, Zhankui Xiao and Feng Wei}

\address{Liang: School of Mathematical Sciences, Huaqiao University,
Quanzhou, Fujian, 362021, P. R. China}

\email{lxfrd@hqu.edu.cn}

\address{Xiao: School of Mathematical Sciences, Huaqiao University,
Quanzhou, Fujian, 362021, P. R. China}

\email{zhkxiao@gmail.com}

\address{Wei: School of Mathematics, Beijing
Institute of Technology, Beijing, 100081, P. R. China}

\email{daoshuo@hotmail.com}

\begin{abstract}
Let $\mathcal{T}$ be a triangular algebra over a commutative ring
$\mathcal{R}$ and $\mathcal{Z(T)}$ be the center of $\mathcal{T}$.
Suppose that ${\mathfrak q}\colon \mathcal{T}\times
\mathcal{T}\longrightarrow \mathcal{T}$ is an $\mathcal{R}$-bilinear
mapping and that ${\mathfrak T}_{\mathfrak q}\colon:
\mathcal{T}\longrightarrow \mathcal{T}$ is a trace of
$\mathfrak{q}$. We describe the form of ${\mathfrak T}_{\mathfrak
q}$ satisfying the condition $[{\mathfrak T}_{\mathfrak
q}(T), T]\in \mathcal{Z(T)}$ for all
$T\in \mathcal{T}$. The question of when ${\mathfrak T}_{\mathfrak
q}$ has the proper form will be addressed. Using the aforementioned
trace function, we establish sufficient conditions for each Lie
triple isomorphism on $\mathcal{T}$ to be almost standard. As
applications we characterize Lie triple isomorphisms of triangular
matrix algebras and nest algebras. Some further research topics related to
current work are proposed at the end of this article.
\end{abstract}



\thanks{The work of the second author is supported by the Mathematical Tianyuan Fundamental of NSFC (Grant No. 11226068).}

\maketitle

\section{Introduction}
\label{xxsec1}

Let $\mathcal{R}$ be a commutative ring with identity, $\mathcal{A}$
be a unital algebra over $\mathcal{R}$ and $\mathcal{Z(A)}$ be the
center of $\mathcal{A}$. Let us denote the commutator or the Lie
product of the elements $a, b\in \mathcal{A}$ by $[a, b]=ab-ba$.
Recall that an $\mathcal{R}$-linear mapping ${\mathfrak f}:
\mathcal{A}\longrightarrow \mathcal{A}$ is said to be
\textit{semi-centralizing} if either $[{\mathfrak f}(a), a]\in
\mathcal{Z(A)}$ or ${\mathfrak f}(a)a+a {\mathfrak f}(a)\in
\mathcal{Z(A)}$ for all $a\in \mathcal{A}$. Further, the mapping
$\mathfrak{f}$ is said to be \textit{centralizing} if $[{\mathfrak
f}(a), a]\in \mathcal{Z(A)}$ for all $a\in \mathcal{A}$. The mapping
$\mathfrak{f}$ is said to be \textit{skew-centralizing} if
${\mathfrak f}(a)a+a{\mathfrak f}(a)\in \mathcal{Z(A)}$ for all
$a\in \mathcal{A}$. In particular, the mapping $\mathfrak{f}$ is
said to be \textit{commuting} if $[{\mathfrak f}(a), a]=0$ for all
$a\in \mathcal{A}$. The mapping $\mathfrak{f}$ is said to be
\textit{skew-commuting} if ${\mathfrak f}(a)a+a{\mathfrak f}(a)=0$
for all $a\in \mathcal{A}$. When we investigate the above-mentioned
mappings, the principal task is to describe their forms. This is
demonstrated by various works, see \cite{BenkovicEremita1,
Bresar1, Bresar2, Bresar3, Bresar4, BresarSemrl, Cheung1, Cheung2,
LeeWongLinWang, LiWei, MarcouxSourour1, Mayne, Posner, Semrl,
Sourour, XiaoWei1, XiaoWei2}. We encourage the reader to read the
well-written survey paper \cite{Bresar4}, in which the author
presented the development of the theory of semi-centralizing
mappings and their applications in details.

Let $\mathcal{R}$ be a commutative ring with identity, $\mathcal{A}$
be a unital algebra over $\mathcal{R}$ and $\mathcal{Z(A)}$ be the
center of $\mathcal{A}$. Recall that an $\mathcal{R}$-linear mapping
${\mathfrak f}: \mathcal{A}\longrightarrow \mathcal{A}$ is said to
be \textit{centralizing} if $[{\mathfrak f}(a), a]\in
\mathcal{Z(A)}$ for all $a\in \mathcal{A}$. Let $n$ be a positive
integer and $\mathfrak{q}\colon \mathcal{A}^n\longrightarrow
\mathcal{A}$ be an $n$-linear mapping. The mapping ${\mathfrak
T}_{\mathfrak q}\colon \mathcal{A}\longrightarrow \mathcal{A}$
defined by ${\mathfrak T}_{\mathfrak q}(a)={\mathfrak q}(a, a,
\cdots, a)$ is called a \textit{trace} of ${\mathfrak q}$. We say
that a centralizing trace ${\mathfrak T}_{\mathfrak q}$ is
\textit{proper} if it can be written as
$$
{\mathfrak T}_{\mathfrak q}(a)=z a^n+\mu_1(a)a^{n-1}+\cdots+\mu_{n-1}(a)a+\mu_n(a)
$$
for all $a\in \mathcal{A}$, where $z\in \mathcal{Z(A)}$ and $\mu_i\ (1\leq i\leq n)$ is a mapping from $\mathcal{A}$ into
$\mathcal{Z(A)}$ and every $\mu_i\ (1\leq i\leq n)$ is in fact a trace
of an $i$-linear mapping ${\mathfrak q}_i$ from $\mathcal{A}^i$ into
$\mathcal{Z(A)}$. Let $n=1$ and ${\mathfrak f}\colon \mathcal{A}\longrightarrow
\mathcal{A}$ be an $\mathcal{R}$-linear mapping. In this case, an
arbitrary trace ${\mathfrak T}_{\mathfrak f}$ of ${\mathfrak f}$
exactly equals to itself. Moreover, if a centralizing trace
${\mathfrak T}_{\mathfrak f}$ of ${\mathfrak f}$ is proper, then it
has the form
$$
{\mathfrak T}_{\mathfrak f}(a)\equiv z a \hspace{4pt} {\rm
mod}\hspace{2pt}\mathcal{Z(A)}, \hspace{8pt} \forall a\in
\mathcal{A},
$$
where $z\in \mathcal{Z(A)}$. Let us see the case of $n=2$. Suppose
that ${\mathfrak g}\colon \mathcal{A}\times
\mathcal{A}\longrightarrow \mathcal{A}$ is an $\mathcal{R}$-bilinear
mapping. If a centralizing trace ${\mathfrak T}_{\mathfrak g}$ of
${\mathfrak g}$ is proper, then it is of the form
$$
{\mathfrak T}_{\mathfrak g}(a)\equiv z a^2+\mu(a)a \hspace{4pt} {\rm
mod}\hspace{2pt}\mathcal{Z(A)}, \hspace{8pt} \forall a\in A,
$$
where $z\in \mathcal{Z(A)}$ and $\mu$ is an $\mathcal{R}$-linear
mapping from $\mathcal{A}$ into $\mathcal{Z(A)}$. It was Bre\v{s}ar who
initiated the study of commuting traces and centralizing traces of
bilinear mappings in his series of works \cite{Bresar1, Bresar2,
Bresar3, Bresar4, BresarSemrl}, where he investigated the structure
of commuting traces and centralizing traces of (bi-)linear mappings
on prime rings. It has turned out that in certain rings, in
particular, prime rings of characteristic different from $2$ and
$3$, every centralizing trace of a biadditive mapping is commuting.
Moreover, every centralizing mapping of a prime ring of
characteristic not $2$ is of the proper form and is actually
commuting. Lee et al further generalized Bre\v{s}ar's results by
showing that each commuting trace of an arbitrary multilinear
mapping on a prime ring also has the proper form
\cite{LeeWongLinWang}.

Cheung in \cite{Cheung2} studied commuting mappings of triangular
algebras (e.g., of upper triangular matrix algebras and nest
algebras). He determined the class of triangular algebras for which
every commuting mapping is proper. Xiao and Wei \cite{XiaoWei1}
extended Cheung's result to the generalized matrix algebra case.
Motivated by the
results of Bre\v{s}ar and Cheung, Benkovi\v{c} and Eremita
\cite{BenkovicEremita1} considered commuting traces of bilinear
mappings on a triangular algebra $
\left[\smallmatrix A & M\\
O & B \endsmallmatrix \right]$. They gave conditions under which
every commuting trace of a triangular algebra $
\left[\smallmatrix A & M\\
O & B \endsmallmatrix \right]$ is proper. In view of the above
works, it is natural and necessary to characterize centralizing
traces of (multi-)linear mappings on triangular algebras. One of the
main aims of this article is to provide a sufficient condition for
each centralizing trace of an arbitrary bilinear mapping on a
triangular algebra $
\left[\smallmatrix A & M\\
O & B \endsmallmatrix \right]$ to be proper.

Another important purpose of this article is to address the Lie
triple isomorphisms problem of triangular algebras. At his 1961 AMS
Hour Talk, Herstein proposed many problems concerning the structure
of Jordan and Lie mappings in associative simple and prime rings
\cite{Herstein}. The renowned Herstein's Lie-type mappings research
program was formulated since then. The involved Lie mappings mainly
include Lie isomorphisms, Lie triple isomorphisms, Lie derivations
and Lie triple derivations et al. Given a commutative ring
$\mathcal{R}$ with identity and two associative
$\mathcal{R}$-algebras $\mathcal{A}$ and $\mathcal{B}$, one define a
\textit{Lie triple isomorphism} from $\mathcal{A}$ into
$\mathcal{B}$ to be an $\mathcal{R}$-linear bijective mapping
${\mathfrak l}$ satisfying the condition
$$
{\mathfrak l}([[a, b],c])=[[{\mathfrak l}(a), {\mathfrak
l}(b)],{\mathfrak l}(c)] \hspace{8pt} \forall a, b, c\in
\mathcal{A}.
$$
For example, an isomorphism or a negative of an anti-isomorphism of
one algebra onto another is also a Lie isomorphism. Furthermore,
every Lie isomorphism and every Jordan isomorphism are Lie triple isomorphisms. One can ask
whether the converse is true in some special cases. That is, does
every Lie triple isomorphism between certain associative algebras
arise from isomorphisms and anti-isomorphisms in the sense of modulo
mappings whose range is central ? Recall that a Lie isomorphism
$\mathfrak{l}\colon A\longrightarrow B$ is \textit{standard} if
$$
\mathfrak{l}=\mathfrak{m}+\mathfrak{n}, \eqno(\clubsuit)
$$
where $\mathfrak{m}$ is an isomorphism or the negative of an anti-isomorphism from
$\mathcal{A}$ onto $\mathcal{B}$ and $\mathfrak{n}: \mathcal{A} \longrightarrow \mathcal{Z(B)}$ is an
$\mathcal{R}$-linear mapping annihilating all commutators.
We say that a Lie triple isomorphism
$\mathfrak{l}\colon A\longrightarrow B$ is \textit{standard} if
$$
\mathfrak{l}=\pm\mathfrak{m}+\mathfrak{n}, \eqno(\spadesuit)
$$
where $\mathfrak{m}$ is an isomorphism or an anti-isomorphism from
$\mathcal{A}$ onto $\mathcal{B}$ and $\mathfrak{n}: \mathcal{A} \longrightarrow \mathcal{Z(B)}$ is an
$\mathcal{R}$-linear mapping annihilating all second commutators.

The resolution of Herstein's Lie isomorphisms problem in matrix
algebra background has been well-known for a long time. Hua
\cite{Hua} proved that every Lie automorphism of the full matrix
algebra $\mathcal{M}_n(\mathcal{D})(n\geq 3)$ over a division ring
$\mathcal{D}$ is of the standard form $(\clubsuit)$. This result was
extended to the nonlinear case by Dolinar \cite{Dolinar} and was
further refined by \v{S}emrl \cite{Semrl}. Dokovi\'{c}
\cite{Dokovic} showed that every Lie automorphism of upper
triangular matrix algebras $\mathcal{T}_n(\mathcal{R})$ over a
commutative ring $\mathcal{R}$ without nontrivial idempotents has
the standard form as well. Marcoux and Sourour
\cite{MarcouxSourour1} classified the linear mappings preserving
commutativity in both directions (i.e., $[x,y] = 0$ if and only if
$[\mathfrak{f}(x), \mathfrak{f}(y)]=0$) on upper triangular matrix
algebras $\mathcal{T}_n(\mathbb{F})$ over a field $\mathbb{F}$. Such
a mapping is either the sum of an algebra automorphism of
$\mathcal{T}_n(\mathbb{F})$ (which is inner) and a mapping into the
center $\mathbb{F}I$, or the sum of the negative of an algebra
anti-automorphism and a mapping into the center $\mathbb{F}I$. The
classification of the Lie automorphisms of
$\mathcal{T}_n(\mathbb{F})$ is obtained as a consequence.
Benkovi\v{c} and Eremita \cite{BenkovicEremita1} applied
the theory of commuting traces to study the Lie isomorphisms on a
triangular algebra. They provided sufficient conditions
under which every commuting trace of triangular algebra $
\left[\smallmatrix A & M\\
O & B \endsmallmatrix \right]$ is proper. It also turns out that under some mild
assumptions, each Lie isomorphism of $
\left[\smallmatrix A & M\\
O & B \endsmallmatrix \right]$ has the standard form $(\clubsuit)$.
Calder\'{o}n Mart\'{i}n and Mart\'{i}n Gonz\'{a}lez observed that
every Lie triple isomorphism of the full matrix algebra
$\mathcal{M}_n(\mathbb{C})$ over the complex field $\mathbb{C}$ is of
the standard form $(\spadesuit)$ \cite{CalderonGonzalez3}.
Simultaneously, Lie triple isomorphisms between rings and between
(non-)self-adjoint operator algebras have received a fair amount of
attentions. The involved rings
and operator algebras include (semi-)prime rings, the algebra of
bounded linear operators, $C^\ast$-algebras, von Neumann algebras,
$H^\ast$-algebras, nest algebras, reflexive algebras and so on, see \cite{CalderonGonzalez1,
CalderonGonzalez2, CalderonGonzalez3, CalderonHaralampidou, Lu1, Lu2,
MarcouxSourour2, Mathieu, Miers1, Miers2, Miers3, QiHou1, QiHou2,
Semrl, Sourour, WangLu, YuLu, ZhangZhang}.

This is the second paper in a series of three that we are planning on
this topic. The first paper was dedicated to studying, in more
details, commuting traces and Lie isomorphisms on generalized matrix
algebras \cite{XiaoWei2}. This article is organized as following.
Section $2$ contains the definition of triangular algebra and some
classical examples. In Section $3$ we provide sufficient conditions
for each centralizing trace of arbitrary bilinear mappings on a
triangular algebra $
\left[\smallmatrix A & M\\
O & B \endsmallmatrix \right]$ to be proper (Theorem
\ref{xxsec3.4}). And then we apply this result to describe the
centralizing traces of bilinear mappings on certain classical
triangular algebras. In Section $4$ we will give sufficient
conditions under which every Lie triple isomorphism from a
triangular algebra into another one has the almost standard form (Theorem
\ref{xxsec4.4}). As corollaries of Theorem \ref{xxsec4.4},
characterizations of Lie triple isomorphisms on several kinds of
triangular algebras are obtained. The last section contains some
potential future research topics related to our current work.

\section{Preliminaries}\label{xxsec2}

Let $\mathcal{R}$ be a commutative ring with identity. Let $A$ and $B$
be unital algebras over $\mathcal{R}$. Recall that an $(A, B)$-bimodule $M$ is
\textit{loyal} if $aMb=0$ implies that $a=0$ or $b=0$ for any $a\in
A, b\in B$. Clearly, each loyal $(A, B)$-bimodule $M$ is faithful as
a left $A$-module and also as a right $B$-module.

Let $A, B$ be unital associative algebras over $\mathcal{R}$ and $M$
be a unital $(A,B)$-bimodule, which is faithful as a left $A$-module
and also as a right $B$-module. We denote the {\em triangular algebra}
consisting of $A, B$ and $M$ by
$$
\mathcal{T}=\left[
\begin{array}
[c]{cc}%
A & M\\
0 & B\\
\end{array}
\right] .
$$
Then $\mathcal{T}$ is an associative and noncommutative
$\mathcal{R}$-algebra. The center $\mathcal{Z(T)}$ of $\mathcal{T}$ is (see \cite[Proposition 3]{Cheung2})
$$
\mathcal{Z(T)}=\left\{ \left[
\begin{array}
[c]{cc}%
a & 0\\
0 & b
\end{array}
\right] \vline \hspace{3pt} am=mb,\ \forall\ m\in M \right\}.
$$
Let us define two natural $\mathcal{R}$-linear projections
$\pi_A:\mathcal{T}\rightarrow A$ and $\pi_B:\mathcal{T}\rightarrow
B$ by
$$
\pi_A: \left[
\begin{array}
[c]{cc}%
a & m\\
0 & b\\
\end{array}
\right] \longmapsto a \quad \text{and} \quad \pi_B: \left[
\begin{array}
[c]{cc}%
a & m\\
0 & b\\
\end{array}
\right] \longmapsto b.
$$
It is easy to see that $\pi_A \left(\mathcal{Z(T)}\right)$ is a
subalgebra of ${\mathcal Z}(A)$ and that $\pi_B(\mathcal{Z(T)})$ is
a subalgebra of ${\mathcal Z}(B)$. Furthermore, there exists a
unique algebraic isomorphism $\tau\colon
\pi_A(\mathcal{Z(T)})\longrightarrow \pi_B(\mathcal{Z(T)})$ such
that $am=m\tau(a)$ for all $a\in \pi_A(\mathcal{Z(T)})$ and for all
$m\in M$.

Let $1$ (resp. $1^\prime$) be the identity of the algebra $A$ (resp.
$B$), and let $I$ be the identity of the triangular algebra
$\mathcal{T}$. We will use the following notations:
$$
P=\left[
\begin{array}
[c]{cc}%
1 & 0\\
0 & 0\\
\end{array}
\right], \hspace{8pt} Q=I-P=\left[
\begin{array}
[c]{cc}%
0 & 0\\
0 & 1^\prime\\
\end{array}
\right]
$$
and
$$
\mathcal{T}_{11}=P{\mathcal T}P, \hspace{6pt}
\mathcal{T}_{12}=P{\mathcal T}Q, \hspace{6pt}
\mathcal{T}_{22}=Q{\mathcal T}Q.
$$
Thus the triangular algebra $\mathcal{T}$ can be written as
$$
\mathcal{T}=P{\mathcal T}P+P{\mathcal T}Q+Q{\mathcal T}Q
=\mathcal{T}_{11}+\mathcal{T}_{12}+\mathcal{T}_{22}.
$$
$\mathcal{T}_{11}$ and $\mathcal{T}_{22}$ are subalgebras of
$\mathcal{T}$ which are isomorphic to $A$ and $B$, respectively.
$\mathcal{T}_{12}$ is a $(\mathcal{T}_{11},
\mathcal{T}_{22})$-bimodule which is isomorphic to the $(A,
B)$-bimodule $M$. It should be remarked that $\pi_A(\mathcal{Z(T)})$
and $\pi_B(\mathcal{Z(T)})$ are isomorphic to $P\mathcal{Z(T)}P$ and
$Q\mathcal{Z(T)}Q$, respectively. Then there is an algebra
isomorphism $\tau\colon P\mathcal{Z(T)}P\longrightarrow
Q\mathcal{Z(T)}Q$ such that $am=m\tau(a)$ for all $m\in
P\mathcal{T}Q$.

Let us list some classical examples of triangular algebras and
matrix algebras which will be revisited in the sequel (Section
\ref{xxsec3}, Section \ref{xxsec4} and Section\ref{xxsec5}). Since
these examples have already been presented in many papers, we just
state their titles without any introduction. We refer the reader to
\cite{BenkovicEremita1, LiWei, XiaoWei1} for more details.
\begin{enumerate}
\item[(\rm a)] Upper and lower triangular matrix algebras;
\item[(\rm b)] Block upper and lower triangular matrix algebras;
\item[(\rm c)] Hilbert space nest algebras;
\item[(\rm d)] Full matrix algebras;
\item[(\rm e)] Inflated algebras.
\end{enumerate}

\section{Centralizing Traces of Triangular Algebras}
\label{xxsec3}

In this section we will establish sufficient conditions for each
commuting trace of arbitrary bilinear mappings on a triangular
algebra $
\left[\smallmatrix A & M\\
O & B \endsmallmatrix \right]$ to be proper (Theorem
\ref{xxsec3.4}). Consequently, we are able to describe centralizing
traces of bilinear mappings on upper triangular matrix algebras and nest algebras. The most
important fact is that Theorem \ref{xxsec3.4} will be used to
characterize Lie triple isomorphisms from a triangular algebra into
another in Section \ref{xxsec4}.

We now list some basic facts related to triangular algebras, which
can be found in \cite[Section 2]{BenkovicEremita1}.

\begin{lemma}\label{xxsec3.1}
Let $M$ be a loyal $(A,B)$-bimodule and let $f, g\colon M
\rightarrow A$ be arbitrary mappings. Suppose $f(m)n+g(n)m=0$ for
all $m, n\in M$. If $B$ is noncommutative, then $f=g=0$.
\end{lemma}

\begin{lemma}\label{xxsec3.2}
Let $\mathcal{T}=\left[\smallmatrix A & M\\
O & B \endsmallmatrix \right]$ be a triangular algebra with a loyal $(A,B)$-bimodule $M$,
$\lambda\in \pi_B(\mathcal{Z(T)})$ and $b\in B$ be a nonzero
element. If $\lambda b=0$, then $\lambda=0$
\end{lemma}

\begin{lemma}\label{xxsec3.3}
Let $\mathcal{T}=\left[\smallmatrix A & M\\
O & B \endsmallmatrix \right]$ be a triangular algebra with a loyal $(A,B)$-bimodule $M$.
Then the center $\mathcal{Z(T)}$ of $\mathcal{T}$ is a domain.
\end{lemma}

We are in position to state the main theorem of this section.

\begin{theorem}\label{xxsec3.4}
Let $\mathcal{T}=\left[\smallmatrix A & M\\
O & B \endsmallmatrix \right]$ be a $2$-torsion free triangular algebra over the commutative
ring $\mathcal{R}$ and ${\mathfrak q}\colon \mathcal{T}\times
\mathcal{T}\longrightarrow \mathcal{T}$ be an $\mathcal{R}$-bilinear
mapping. If
\begin{enumerate}
\item[(1)] each commuting linear mapping on $A$ or $B$ is proper,
\item[(2)] $\pi_A(\mathcal{Z(T)})={\mathcal Z}(A) \neq A $ and
$\pi_B(\mathcal{Z(T)})={\mathcal Z}(B)\neq B$,
\item[(3)] $M$ is loyal,
\end{enumerate}
then every centralizing trace ${\mathfrak T}_{\mathfrak q}:
\mathcal{T}\longrightarrow \mathcal{T}$ of ${\mathfrak q}$ is
proper.
\end{theorem}

For convenience, let us write $A_{1}={A}$, $A_{2}=B$ and $A_{3}=M$.
We denote the unity of $A_{1}$ by $1$ and the unity of $A_{2}$ by
$1'$. Suppose that $\mathfrak{T_q}$ is an arbitrary trace of the
$\mathcal{R}$-bilinear mapping $\mathfrak{q}$. Then there exist
bilinear mappings $f_{ij}\colon A_{i}\times A_{j} \rightarrow
A_{1}$, $g_{ij}\colon A_{i}\times A_{j} \rightarrow A_{2}$ and
$h_{ij}\colon A_{i}\times A_{j} \rightarrow A_{3}$ ($1\leqslant
i\leqslant j\leqslant 3$) such that
$$
\mathfrak{T_q}\colon \left[\begin{array}{cc}
a_{1} & a_{3}\\
 & a_{2}\end{array}\right]\mapsto
 \left[\begin{array}{cc}
F(a_{1},a_{2},a_{3}) & H(a_{1},a_{2},a_{3})\\
 & G(a_{1},a_{2},a_{3})\end{array}\right],
$$
where
$$
F(a_{1},a_{2},a_{3})=\sum_{1\leqslant i\leqslant j\leqslant 3}f_{ij}(a_{i},a_{j}),
$$
$$
G(a_{1},a_{2},a_{3})=\sum_{1\leqslant i\leqslant j\leqslant 3}g_{ij}(a_{i},a_{j}),
$$
$$
H(a_{1},a_{2},a_{3})=\sum_{1\leqslant i\leqslant j\leqslant 3}h_{ij}(a_{i},a_{j}).
$$
Since $\mathfrak{T_q}$ is centralizing, we have
$$
\left[\left[
\begin{array}[c]{cc}%
F & H\\
& G\end{array}\right],\left[
\begin{array}[c]{cc}%
a_{1} & a_{3}\\
& a_{2}\end{array}
\right]\right]=\left[\begin{array}[c]{cc}%
[F,a_{1}] & Fa_{3}+Ha_{2}-a_{1}H-a_{3}G\\
& [G,a_{2}]\end{array}\right]\in \mathcal{Z(T)}. \eqno(3.1)
$$

Now we divide the proof of Theorem \ref{xxsec3.4} into a series of
lemmas for comfortable reading.

\begin{lemma}\label{xxsec3.5}
Let $K: A_{2}\times A_{2}\rightarrow A_{3}$ (resp. $K: A_{1}\times
A_{1}\rightarrow A_{3}$) be an $\mathcal{R}$-bilinear mapping. If
$K(x,x)x=0$ {\rm (}resp. $xK(x,x)=0$ {\rm )} for all $x\in A_2$ {\rm
(}resp. for all $x\in A_1${\rm )}, then $K(x,x)=0$.
\end{lemma}

\begin{proof}
Setting $x=1'$, we obtain that $K(1',1')=0$. Replacing $x$ by $x+1'$
in $K(x,x)x=0$, we get
$$
K(x,x)=-(K(1',x)+K(x,1'))(1'+x). \eqno(3.2)
$$
Substituting $x-1'$ for $x$ in $K(x,x)x=0$, we arrive at
$$
K(x,x)=(K(1',x)+K(x,1'))(1'-x).  \eqno(3.3)
$$
Combining the above two relations gives $K(1',x)+K(x,1')=0$. Thus
$K(x,x)=0$.
\end{proof}

\begin{lemma}\label{xxsec3.6}
$H(a_{1},a_{2},a_{3})=h_{13}(a_{1},a_{3})+h_{23}(a_{2},a_{3})+
h_{33}(a_{3},a_{3})$.
\end{lemma}

\begin{proof}
It follows from the matrix relation $(3.1)$ that
$$
Fa_{3}+Ha_{2}-a_{1}H-a_{3}G=0. \eqno(3.4)
$$
Let us take $a_{1}=0$ and $a_{2}=0$ into $(3.4)$. Then $(3.1)$
implies that
$$
f_{33}(a_{3},a_{3})a_{3}=a_{3}g_{33}(a_{3},a_{3}) \eqno(3.5)
$$
for all $a_3\in A_3$. Let us choose $a_1=0$ and $a_3=0$ in $(3.4)$.
Then $0=Ha_{2}=h_{22}(a_{2}, a_{2})a_{2}$ for all $a_{2}\in A_{2}$.
In view of Lemma \ref{xxsec3.5}, we have $h_{22}(a_{2},a_{2})=0$.
Similarly, putting $a_{2}=0$ and $a_{3}=0$ in $(3.4)$ yields
$h_{11}(a_{1},a_{1})=0$ for all $a_{1}\in A_{1}$. Furthermore,
setting $a_3=0$ in $(3.4)$, we see that
$$
(h_{12}(a_1,a_2)a_2-a_1h_{12}(a_1,a_2))=0
$$
for all $a_1\in A_1, a_2\in A_2$. Replacing $a_1$ by $-a_1$ in the
above relation and comparing the obtained two relations gives
$a_1h_{12}(a_1,a_2)=0$ for all $a_1\in A_1, a_2\in A_2$. In
particular, $h_{12}(1,a_2)=0$ for all $a_2\in A_2$. Substituting
$a_1+1$ for $a_1$ in $a_1h_{12}(a_1,a_2)=0$ leads to
$h_{12}(a_1,a_2)=0$ for all $a_1\in A_1, a_2\in A_2$. Therefore
$$
H(a_{1},a_{2},a_{3})=h_{13}(a_{1},a_{3})+h_{23}(a_{2},a_{3})+ h_{33}(a_{3},a_{3})
$$
for all $a_1\in A_1, a_2\in A_2, a_3\in A_3$.
\end{proof}

\begin{lemma}\label{xxsec3.7}
With notations as above, we have
\begin{enumerate}
\item[(1)]
$a_1\mapsto f_{11}(a_1,a_1)$ is a commuting trace,\\
$a_1\mapsto f_{13}(a_1,a_3)$ is a commuting linear mapping for each $a_3\in A_3$,\\
$a_2\mapsto g_{22}(a_2,a_2)$ is a commuting trace,\\
$a_2\mapsto g_{23}(a_2,a_3)$ is a commuting linear mapping for each
$a_3\in A_3$,
\item[(2)]
$[g_{11}(a_1,a_1),a_2]=\tau([f_{12}(a_{1},a_{2}),a_{1}])\in {\mathcal Z}(A_{2})$, \\
$[g_{12}(a_1,a_2),a_2]=\tau([f_{22}(a_{2},a_{2}),a_{1}])\in {\mathcal Z}(A_{2})$, \\
$[g_{13}(a_1,a_3),a_2]=\tau([f_{23}(a_{2},a_{3}),a_{1}])\in
{\mathcal Z}(A_{2})$,
\item[(3)]
$f_{33}(a_3, a_3)\in Z(A_{1})$ and $g_{33}(a_3, a_3)\in {\mathcal
Z}(A_{2})$.
\end{enumerate}
\end{lemma}

\begin{proof}
By the relation $(3.1)$ we know that
$$
\tau([F,a_{1}])=[G,a_{2}]. \eqno(3.6)
$$
Let us take $a_1=0$ in $(3.6)$. Then
$$
[g_{22}(a_2,a_2)+g_{23}(a_2,a_3)+g_{33}(a_3,a_3),a_2]=0 \eqno(3.7)
$$
for all $a_2\in A_2, a_3\in A_3$. Replacing $a_3$ by $-a_3$ in
$(3.7)$ we get
$$
[g_{22}(a_2,a_2)+g_{33}(a_3,a_3),a_2]=0 \eqno(3.8)
$$
for all $a_2\in A_2, a_3\in A_3$. Putting $a_3=0$ in $(3.7)$ and
combining $(3.7)$ and $(3.8)$, we obtain
$$
[g_{22}(a_2,a_2),a_2]=0,\quad [g_{23}(a_2,a_3),a_2]=0,\quad [g_{33}(a_3,a_3),a_2]=0
$$
for all $a_2\in A_2, a_3\in A_3$. In a similar way, we have
$$
[f_{11}(a_1,a_1),a_1]=0,\quad [f_{33}(a_3,a_3),a_1]=0,\quad [f_{13}(a_1,a_3),a_1]=0.
$$

Setting $a_{3}=0$ in $(3.6)$, we arrive at
$$
\tau([f_{12}(a_1,a_2)+f_{22}(a_2,a_2),a_1])=[g_{11}(a_1,a_1)+g_{12}(a_1,a_2),a_2]
\eqno(3.9)
$$
for all $a_1\in A_1, a_2\in A_2$. Replacing $a_1$ by $-a_1$ in
$(3.9)$ and then comparing the obtained relation with $(3.9)$, we
get
$$
\tau([f_{22}(a_{2},a_{2}),a_{1}])=[g_{12}(a_{1},a_{2}),a_{2}]\in
{\mathcal Z}(A_{2})  \eqno(3.10)
$$
and
$$
\tau([f_{12}(a_{1},a_{2}),a_{1}])=[g_{11}(a_{1},a_{1}),a_{2}]\in
{\mathcal Z}(A_{2})  \eqno(3.11)
$$
for all $a_{1}\in A_{1}, a_{2}\in A_{2}$. In view of $(3.6), (3.10),
(3.11)$ we conclude
$$
\tau([f_{23}(a_{2},a_{3}),a_{1}])=[g_{13}(a_{1},a_{3}),a_{2}]\in
{\mathcal Z}(A_{2})
$$
for all $a_{1}\in A_{1}, a_{2}\in A_{2}, a_{3}\in A_{3}$.
\end{proof}

\begin{lemma}\label{xxsec3.8}
There exist a linear mapping $\xi: A_3\rightarrow {\mathcal Z}(A_2)$
and a bilinear mapping $\eta: A_2\times A_3\rightarrow {\mathcal
Z}(A_2)$ such that $g_{23}(a_2,a_3)=\xi(a_3)a_2+\eta(a_2,a_3)$.
\end{lemma}

\begin{proof}
Since $a_2\mapsto g_{23}(a_2,a_3)$ is a commuting linear mapping for
each $a_3\in A_3$, then by the hypothesis $(1)$ there exist mappings
$\xi: A_{3}\rightarrow {\mathcal Z}(A_2)$ and $\eta: A_2\times
A_3\rightarrow {\mathcal Z}(A_2)$ such that
$$
g_{23}(a_2,a_3)=\xi(a_3)a_{2}+\eta(a_2,a_3),
$$
where $\eta$ is $\mathcal{R}$-linear in the first argument. Let us
show that $\xi$ is $\mathcal{R}$-linear and $\eta$ is
$\mathcal{R}$-bilinear. Clearly,
$$
\begin{aligned}
g_{23}(a_2,a_3+b_3)&=\xi(a_3+b_3)a_{2}+\eta(a_2,a_3+b_3)\\
g_{23}(a_2,a_3)+g_{23}(a_2,b_3)&=\xi(a_3)a_2+\eta(a_2,a_3)+\xi(b_3)a_2+\eta(a_2,b_3)
\end{aligned}
$$
for all $a_2\in A_2, a_3, b_3\in A_3$. So
$$
\big(\xi(a_3+b_3)-\xi(a_3)-\xi(b_3)\big)a_2
+\eta(a_{2},a_{3}+b_{3})-\eta(a_{2},a_{3})-\eta(a_{2},b_{3})=0
$$
for all $a_2\in A_2, a_3, b_3\in A_3$. Note that $\xi$ and $\eta$
map into $Z(A_{2})$. Hence
$(\xi(a_{3}+b_{3})-\xi(a_{3})-\xi(b_{3}))[a_2,b_2]=0$ for all $a_2,
b_2\in  A_2$, and $a_3, b_3\in  A_3$. Note that $A_2$ is noncommutative. Applying Lemma \ref{xxsec3.2}
yields that $\xi$ is $\mathcal{R}$-linear mapping. Consequently,
$\eta$ is $\mathcal{R}$-linear in the second argument.
\end{proof}

\begin{lemma}\label{xxsec3.9}
$f_{23}(a_2,a_3)\in {\mathcal Z}(A_1)$ and $g_{13}(a_1,a_3)\in
{\mathcal Z}(A_2)$.
\end{lemma}

\begin{proof}
By Lemma \ref{xxsec3.7} it is enough to prove $f_{23}(a_2,a_3)\in
{\mathcal Z}(A_1)$. Setting $a_1=0$ in $(3.4)$ and using $(3.5)$, we
obtain
$$
\begin{aligned}
&\big(f_{22}(a_2,a_2)+f_{23}(a_2,a_3)\big)a_3+\big(h_{33}(a_3,a_3)+h_{23}(a_2,a_3)\big)a_2\\
&\quad -a_3\big(g_{22}(a_2,a_2)+g_{23}(a_2,a_3)\big)=0
\end{aligned}\eqno(3.12)
$$
for all $a_2\in A_2, a_3\in A_3$. Replacing $a_2$ by $-a_2$ in the
equation $(3.12)$ and then comparing with it, we get
$$
h_{23}(a_{2},a_{3})a_{2}=a_{3}g_{22}(a_{2},a_{2})-f_{22}(a_{2},a_{2})a_{3}
\eqno(3.13)
$$
and
$$
h_{33}(a_{3},a_{3})a_{2}=a_{3}g_{23}(a_{2},a_{3})-f_{23}(a_{2},a_{3})a_{3}
\eqno(3.14)
$$
for all $a_2\in A_2, a_3\in A_3$. Note that
$[g_{23}(a_2,a_3),a_2]=0$ for all $a_2\in A_2, a_3\in  A_3$.
Replacing $a_2$ by $a_2+1'$ in $[g_{23}(a_2,a_3),a_2]=0$ gives
$g_{23}(1',a_3)\in {\mathcal Z}(A_2)$. On the other hand, Lemma
\ref{xxsec3.7} shows that $f_{23}(1',a_3)\in {\mathcal Z}(A_1)$ for
all $a_3\in A_3$. Taking $a_2=1'$ in $(3.14)$ we have
$$
h_{33}(a_3,a_3)=a_3\alpha(a_3), \eqno(3.15)
$$
where $\alpha(a_3)=g_{23}(1',a_3)-\tau(f_{23}(1',a_3))\in {\mathcal
Z}(A_2)$. It follows from $(3.14)$, $(3.15)$ and Lemma
\ref{xxsec3.8} that
$$
a_3(\alpha(a_3)-\xi(a_3))a_2=\big(\tau^{-1}(\eta(a_2,a_3))-f_{23}(a_2,a_3)\big)a_3.\eqno(3.16)
$$
We denote $Y(a_3)=\alpha(a_3)-\xi(a_3)$,
$X(a_2,a_3)=\tau^{-1}(\eta(a_2,a_3))-f_{23}(a_2,a_3)$. Taking
$a_2=1'$ into $(3.16)$, we see that
$(\tau^{-1}(Y(a_3))-X(1',a_3))a_3=0$ for all $a_3\in A_3$.

We claim that
$$
Y(a_3)=\tau(X(1',a_3))\eqno(3.17)
$$
for all $a_3\in A_3$. In fact,
replacing $a_3$ by $m+n$ in $(\tau^{-1}(Y(a_3))-X(1',a_3))a_3=0$, we get
$$
(\tau^{-1}(Y(m))-X(1',m))n+(\tau^{-1}(Y(n))-X(1',n))m=0
$$
for all $m, n\in A_3$. Applying Lemma \ref{xxsec3.1} yields
$Y(m)=\tau(X(1',m))$ for all $m\in A_3$. Thus our claim follows.

Now let us rewrite the relation $(3.16)$ as
$$
a_3\tau(X(1',a_3))a_2=X(a_2,a_3)a_3\eqno(3.18)
$$
for all $a_3\in A_3$. Replacing $a_3$ by $m+n$ in $(3.18)$, we
obtain
$$
m\tau(X(1',n))a_2+n\tau(X(1',m))a_2=X(a_2,n)m+X(a_2,m)n  \eqno(3.19)
$$
for all $a_2\in A_2$, $m, n\in A_3$. Replacing $n$ by $a_1n$ in
$(19)$ and then subtracting the left multiplication of $(3.19)$ by $a_1$, we
arrive at
$$
\begin{aligned}
&m\tau(X(1',a_1n))a_2-a_1m\tau(X(1',n))a_2 \\
&=X(a_2,m)a_1n+X(a_2,a_1n)m-a_1X(a_2,m)n-a_1X(a_2,n)m
\end{aligned} \eqno(3.20)
$$
for all $a_1\in A_1$, $a_2\in A_2$ and $m, n\in A_3$. Taking $m=n$
in $(3.20)$ and using $(3.18)$, we have
$$
m\tau(X(1',a_1m))a_2=X(a_2,a_1m)m+[X(a_2,m), a_1]m \eqno(3.21)
$$
for all $a_1\in A_1, a_2\in A_2, m\in A_3$. Left multiplying $a_1$
in $(3.21)$ and considering $(3.18)$, we get $[X(a_2,a_1m),
a_1]m=a_1[X(a_2,m), a_1]m$. That is,
$$
([X(a_2,a_1m),
a_1]-a_1[X(a_2,m), a_1])m=0
$$
for all $a_1\in A_1$, $a_2\in A_2$ and $m\in A_3$. Let us write
$P(m)=[X(a_2,a_1m), a_1]-a_1[X(a_2,m), a_1]$ for some fixed $a_1,
a_2$. Then $P\colon A_3\rightarrow A_1$ is an $\mathcal{R}$-linear
mapping for each $a_1\in A_1, a_2\in A_2$, and $P(m)m=0$. A
linearization of $P(m)m=0$ shows $P(m)n+P(n)m=0$ for all $m, n\in
A_3$. In view of Lemma \ref{xxsec3.1} we know that $P(m)=0$. So
$$
[X(a_{2},a_{1}m), a_{1}]=a_{1}[X(a_{2},m), a_{1}]
$$
for all $a_1\in A_1, a_2\in A_2, m\in A_3$. Picking $b_1\in A_1$
such that $[a_1, b_1]\neq 0$, and then commuting with $b_{1}$, we
get $[a_1, b_1][X(a_2,m), a_{1}]=0$ since $[X(a_2,m), a_{1}]=[a_1,
f_{23}(a_2, m)]\in {\mathcal Z}(A_1)$ by Lemma \ref{xxsec3.7}. Thus
Lemma \ref{xxsec3.2} implies $[X(a_{2},m),a_{1}]=[a_1,
f_{23}(a_2,a_3)]=0$ and this completes the proof of the lemma.
\end{proof}

\begin{lemma}\label{xxsec3.10}
With notations as above, we have
\begin{enumerate}
\item[(1)]
$f_{22}(a_2, a_2)\in {\mathcal Z}(A_1)$ and $g_{11}(a_1,a_1)\in
{\mathcal Z}(A_2)$;
\item[(2)]
$a_1\mapsto f_{12}(a_1,a_2)$ is a commuting linear mapping for each $a_2\in A_2$,\\
$a_2\mapsto g_{12}(a_1,a_2)$ is a commuting linear mapping for each
$a_1\in A_1$.
\end{enumerate}
\end{lemma}

\begin{proof}
Taking $a_{2}=0$ in $(3.4)$ and using $(3.5)$, we get
$$
\begin{aligned}
&(f_{11}(a_{1},a_{1})+f_{13}(a_{1},a_{3}))a_{3}-a_{3}(g_{11}(a_{1},a_{1})+g_{13}(a_{1},a_{3}))\\
&\quad -a_{1}(h_{13}(a_{1},a_{3})+h_{33}(a_{3},a_{3}))=0
\end{aligned}\eqno(3.22)
$$
for all $a_1\in A_1, a_3\in A_3$. Note that $\mathcal{R}$ is
$2$-torsion free ring. Substituting $-a_1$ for $a_1$ in $(3.22)$, we
obtain
$$
a_1h_{13}(a_1,a_3)=f_{11}(a_1,a_1)a_3-a_3g_{11}(a_1,a_1) \eqno(3.23)
$$
for all $a_1\in A_1, a_3\in A_3$. Combining $(3.22)$ with $(3.23)$
gives
$$
a_1h_{33}(a_3,a_3)=f_{13}(a_1,a_3)a_3-a_3g_{13}(a_1,a_3) \eqno(3.24)
$$
for all $a_1\in A_1, a_3\in A_3$. On the other hand, replacing $a_3$
by $a_1a_3$ in $(3.13)$ and subtracting the left multiplication of $(3.13)$ by
$a_1$ we get
$$
(a_1h_{23}(a_2,a_3)-h_{23}(a_2,a_1a_3))a_2=[f_{22}(a_2,a_2),a_1]a_3
\eqno(3.25)
$$
for all $a_1\in A_1, a_2\in A_2, a_3\in A_3$. Replacing $a_3$ by
$a_3a_2$ in $(3.13)$ and subtracting the right multiplication of $(3.11)$ by $a_2$ we get
$h_{23}(a_2,a_3a_2)a_2=h_{23}(a_2,a_3)a_2a_2$. Let us set
$K(x,y)=h_{23}(x,a_3y)-h_{23}(x,a_3)y$, where $x,y\in A_2$. It is
easy to see that $K(x,y)\colon A_2\times A_2\rightarrow A_3$ is an
$\mathcal{R}$-bilinear mapping, and $K(a_2,a_2)a_2=0$. It follows
from Lemma \ref{xxsec3.5} that
$$
h_{23}(a_2,a_3a_2)=h_{23}(a_2,a_3)a_{2} \eqno(3.26)
$$
for all $a_2\in A_2, a_3\in A_3$. Substituting $a_3a_2$ for $a_3$ in
$(3.23)$ and then subtracting the right multiplication of $(3.23)$ by $a_2$, we
have
$$
a_{3}[g_{11}(a_{1},a_{1}),a_{2}]=a_{1}(h_{13}(a_{1},a_{3}a_{2})-h_{13}(a_{1},a_{3})a_{2})
\eqno(3.27)
$$
for all $a_1\in A_1, a_2\in A_2$ and $a_3\in A_3$. Combining the
relations $(3.13)-(3.14)$, $(3.23)-(3.24)$ together with $(3.4)$
yields
$$
a_{1}h_{23}(a_{2},a_{3})+a_{3}g_{12}(a_{1},a_{2})=h_{13}(a_{1},a_{3})a_{2}+f_{12}(a_{1},a_{2})a_{3}
\eqno(3.28)
$$
for all $a_1\in A_1, a_2\in A_2$ and $a_3\in A_3$. Replacing $a_3$
by $a_3a_2$ in $(3.28)$ and then subtracting the right multiplication of $(3.28)$ by $a_2$, we arrive at
$$
\begin{aligned}
&a_{1}(h_{23}(a_{2},a_{3})a_{2}-h_{23}(a_{2},a_{3}a_{2}))+a_{3}[g_{12}(a_{1},a_{2}),a_{2}] \\
&\quad =(h_{13}(a_{1},a_{3})a_{2}-h_{13}(a_{1},a_{3}a_{2}))a_{2}
\end{aligned} \eqno(3.29)
$$
for all $a_1\in A_1, a_2\in A_2$ and $a_3\in A_3$. Considering the
identities $(3.26)$ and $(3.29)$, we get
$$
-a_{3}[g_{12}(a_{1},a_{2}),a_{2}]=(h_{13}(a_{1},a_{3}a_{2})-h_{13}(a_{1},a_{3})a_{2})a_{2} \eqno(3.30)
$$
for all $a_1\in A_1, a_2\in A_2$ and $a_3\in A_3$. Making the right multiplication of
$(3.27)$ by $a_2$ and then subtracting the left multiplication of $(3.30)$ by
$a_1$, we obtain
$$
a_{1}a_{3}[g_{12}(a_{1},a_{2}),a_{2}]=a_{3}[a_{2},g_{11}(a_{1},a_{1})]a_{2}
$$
for all $a_1\in A_1, a_2\in A_2$ and $a_3\in A_3$. According to
$(3.10)$, we have
$$
a_{1}[f_{22}(a_{2},a_{2}),a_{1}]a_{3}=a_{3}[a_{2},g_{11}(a_{1},a_{1})]a_{2}
$$
for all $a_1\in A_1, a_2\in A_2$ and $a_3\in A_3$. Therefore
$$
\left[
\begin{array}
[c]{cc}%
a_{1}[f_{22}(a_{2},a_{2}),a_{1}] & 0\\
0 & [a_{2},g_{11}(a_{1},a_{1})]a_{2}\\
\end{array}
\right] \in \mathcal{Z(T)}.
$$
Commuting with $b_2\in A_2$, we get
$[g_{11}(a_1,a_1),a_2][a_2,b_2]=0$. Then Lemma \ref{xxsec3.2}
implies $g_{11}(a_1,a_1)\in {\mathcal Z}(A_{2})$ and hence
$a_1\mapsto f_{12}(a_1,a_2)$ is a commuting linear mapping for each
$a_2\in A_2$ by Lemma \ref{xxsec3.7}. Similarly, we have
$f_{22}(a_2, a_2)\in {\mathcal Z}(A_1)$ and $a_2\mapsto
g_{12}(a_1,a_2)$ is a commuting linear mapping for each $a_1\in
A_1$.
\end{proof}

{\noindent}{\bf Proof of Theorem 3.4.} Let ${\mathfrak q}\colon
\mathcal{T}\times \mathcal{T}\longrightarrow \mathcal{T}$ be an
arbitrary $\mathcal{R}$-bilinear mapping of $\mathcal{T}$. It
follows from Lemma \ref{xxsec3.7}, Lemma \ref{xxsec3.9} and Lemma
\ref{xxsec3.10} that every centralizing trace of $\mathfrak{q}$ is
commuting. Then the desired result can be obtained by \cite[Theorem
3.1]{BenkovicEremita1}. \qed \vspace{2mm}

An algebra $\mathcal{A}$ over a commutative ring $\mathcal{R}$ is said to be {\em central}
over $\mathcal{R}$ if $\mathcal{Z(A)}=\mathcal{R}1$. The following technical lemma
will be used to deal with the centralizing traces of upper triangular matrix algebras.

\begin{lemma}\label{xxsec3.11}
Let $\mathcal{T}=\left[\smallmatrix \mathcal{R} & M\\
O & B \endsmallmatrix \right]$ be a $2$-torsion free triangular algebra over the commutative
ring $\mathcal{R}$ and ${\mathfrak q}\colon \mathcal{T}\times
\mathcal{T}\longrightarrow \mathcal{T}$ be an $\mathcal{R}$-bilinear
mapping. Suppose that $B$ is noncommutative and both
$\mathcal{T}$ and $B$ are central over
$\mathcal{R}$. If
\begin{enumerate}
\item[(1)] each commuting linear mapping on $B$ is proper,
\item[(2)] for any $r\in \mathcal{R}$ and $m\in M$,
$rm=0$ implies $r=0$ or $m=0$,
\item[(3)] there exist $m_0\in M$ and $b_0\in B$ such that
$m_0b_0$ and $m_0$ are linearly independent over $\mathcal{R}$,
\end{enumerate}
then each centralizing trace ${\mathfrak T}_{\mathfrak q}:
\mathcal{T}\longrightarrow \mathcal{T}$ of ${\mathfrak q}$ is
proper.
\end{lemma}

\begin{proof}
We use the same notations of Theorem \ref{xxsec3.4}. Since $A_1=\mathcal{R}$ is commutative,
then the equation $(3.1)$ shows that $[F, a_1]=0$ and hence $[G, a_2]=0$. Therefore the
centralizing trace ${\mathfrak T}_{\mathfrak q}$ is commuting. Now the desired result
follows from \cite[Lemma 3.2]{BenkovicEremita1}.
\end{proof}

\begin{corollary}\label{xxsec3.12}
Let $\mathcal{R}$ be a $2$-torsion free commutative domain
and $\mathcal{T}_n(\mathcal{R})(n\geq 2)$ be the algebra of all $n\times n$ upper triangular matrices over $\mathcal{R}$. Suppose that ${\mathfrak q}\colon \mathcal{T}_n(\mathcal{R})\times
\mathcal{T}_n(\mathcal{R})\longrightarrow \mathcal{T}_n(\mathcal{R})$ is an $\mathcal{R}$-bilinear
mapping. Then every centralizing trace ${\mathfrak T}_{\mathfrak q}:
\mathcal{T}_n(\mathcal{R})\longrightarrow \mathcal{T}_n(\mathcal{R})$ of ${\mathfrak q}$ is
proper.
\end{corollary}

\begin{proof}
The proof is similar with that of \cite[Corollary 3.4]{BenkovicEremita1} and hence we omit it here.
\end{proof}

Applying Theorem \ref{xxsec3.4} and \cite[Corollary 3.5]{BenkovicEremita1}
yields

\begin{corollary}\label{xxsec3.13}
Let $\textbf{\rm H}$ be a Hilbert space, $\mathcal{N}$ be a nest of $\textbf{\rm H}$ and ${\mathcal Alg}(\mathcal{N})$ be the nest algebra associated with $\mathcal{N}$. Suppose that ${\mathfrak q}\colon {\mathcal Alg}(\mathcal{N})\times
{\mathcal Alg}(\mathcal{N})\longrightarrow {\mathcal Alg}(\mathcal{N})$ is an $\mathcal{R}$-bilinear
mapping. Then every centralizing trace ${\mathfrak T}_{\mathfrak q}:
{\mathcal Alg}(\mathcal{N})\longrightarrow {\mathcal Alg}(\mathcal{N})$ of ${\mathfrak q}$ is
proper.
\end{corollary}

\section{Lie Triple Isomorphisms on Triangular Algebras}\label{xxsec4}

\begin{lemma}\label{xxsec4.1}
Let $\mathcal{R}$ be $2$-torsion free. Then the triangular algebra
$\mathcal{T}=\left[\smallmatrix A & M\\
O & B \endsmallmatrix \right]$ does not contain nonzero central Jordan ideals.
\end{lemma}

\begin{proof}
Let $\mathcal{J}$ be a central Jordan ideal of $\mathcal{T}$.
Suppose that $\left[\smallmatrix \alpha & 0\\
0 & \tau(\alpha) \endsmallmatrix \right]\in \mathcal{J}$. Hence
$$
\left[\begin{array}{cc}
\alpha & 0\\
& \tau(\alpha)\end{array}\right]\circ \left[\begin{array}{cc}
0 & m\\
& 0\end{array}\right]=\left[\begin{array}{cc}
0 & \alpha m+m\tau(\alpha)\\
& 0\end{array}\right]
$$
for all $m\in M$. This implies that $2\alpha M=0$ and so
$\alpha=0=\left[\smallmatrix \alpha & 0\\
0 & \tau(\alpha) \endsmallmatrix \right]$.
\end{proof}

\begin{theorem}\label{xxsec4.2}
Let $\mathcal{T}=\left[\smallmatrix A & M\\
O & B \endsmallmatrix \right]$ and $\mathcal{T^\prime}=\left[\smallmatrix A^\prime & M^\prime\\
O & B^\prime \endsmallmatrix \right]$ be two triangular algebras over a commutative ring
$\mathcal{R}$ with $\frac{1}{2}\in \mathcal{R}$ and let
$\mathfrak{l}:\mathcal{T}\longrightarrow \mathcal{T^\prime}$ be a
Lie triple isomorphism. If
\begin{enumerate}
\item[(1)] each centralizing trace of a bilinear mapping on $\mathcal{T^\prime}$ is proper,
\item[(2)] at least one of $A, B$ and at least one of  $A^\prime, B^\prime$ are noncommutative,
\item[(3)] $M^\prime$ is loyal,
\end{enumerate}
then $\mathfrak{l}=\pm \mathfrak{m}+\mathfrak{n}$, where
$\mathfrak{m}\colon \mathcal{T}\rightarrow \mathcal{T'}$ is a Jordan
homomorphism, $\mathfrak{m}$ is one-to-one, and $\mathfrak{n}\colon
\mathcal{T}\longrightarrow \mathcal{Z(T^\prime)}$ is a linear
mapping vanishing on each second commutator. Moreover, if
$\mathcal{T^\prime}$ is central over $\mathcal{R}$, then
$\mathfrak{m}$ is onto.
\end{theorem}

\begin{proof}
For arbitrary $x,z\in \mathcal{T}$, it is easy to see that
$\mathfrak{l}$ satisfies
$[[\mathfrak{l}(x^{2}),\mathfrak{l}(x)],\mathfrak{l}(z)]=\mathfrak{l}([[x^{2},x],z])=0$.
Since $\mathfrak{l}$ is onto,
$[\mathfrak{l}(x^{2}),\mathfrak{l}(x)]\in \mathcal{Z(T^\prime)}$ for
all $x\in \mathcal{T}$. Replacing $x$ by $\mathfrak{l}^{-1}(y)$, we
get $[\mathfrak{l}(\mathfrak{l}^{-1}(y)^{2}),y]\in
\mathcal{Z(T^\prime)}$ for all $y\in \mathcal{T^\prime}$. This means
that the mapping
$\mathfrak{T_q}(y)=\mathfrak{l}(\mathfrak{l}^{-1}(y)^{2})$ is
centralizing. Since $\mathfrak{T_q}$ is also a trace of the bilinear
mapping $\mathfrak{q}\colon \mathcal{T'}\times \mathcal{T^\prime}
\longrightarrow \mathcal{T^\prime}$,
$\mathfrak{q}(y,z)=\mathfrak{l}(\mathfrak{l}^{-1}(y)\mathfrak{l}^{-1}(z))$,
by the hypothesis $(1)$ there exist $\lambda\in
\mathcal{Z(T^\prime)}$, a linear mapping $\mu_{1}:
\mathcal{T^\prime}\longrightarrow \mathcal{Z(T^\prime)}$, and a
trace $\nu_{1}: \mathcal{T^\prime}\longrightarrow
\mathcal{Z(T^\prime)}$ of a bilinear mapping such that
$$
\mathfrak{l}(\mathfrak{l}^{-1}(y)^{2})=\lambda
y^{2}+\mu_{1}(y)y+\nu_{1}(y) \eqno(4.1)
$$
for all $y\in \mathcal{T'}$. Let $\mu=\mu_1\mathfrak{l}$ and
$\nu=\nu_1\mathfrak{l}$. Then $\mu$ and $\nu$ are mappings of
$\mathcal{T}$ into $\mathcal{Z(T^\prime)}$ and $\mu$ is linear.
Hence $(4.1)$ can be rewritten as
$$
\mathfrak{l}(x^{2})=\lambda\mathfrak{l}(x)^{2}+\mu(x)\mathfrak{l}(x)+\nu(x)
\eqno(4.2)
$$
for all $x\in \mathcal{T}$. We conclude that $\lambda\neq 0$.
Otherwise, we have $\mathfrak{l}(x^{2})-\mu(x)\mathfrak{l}(x)\in
\mathcal{Z(T^\prime)}$ by $(4.2)$ and hence
$$
\begin{aligned}
\mathfrak{l}([[x^2,y],[x,y]])
&=[[\mathfrak{l}(x^2),\mathfrak{l}(y)],\mathfrak{l}([x,y])]\\
&=[[\mu(x)\mathfrak{l}(x),\mathfrak{l}(y)],\mathfrak{l}([x,y])] \\
&=\mu(x)[[\mathfrak{l}(x),\mathfrak{l}(y)],\mathfrak{l}([x,y])]\\
&=\mu(x)\mathfrak{l}([[x,y],[x,y]])\\
&=0
\end{aligned}
$$
for all $x,y\in \mathcal{T}$. Consequently, $[[x^2,y],[x,y]]=0$ for
all $x,y\in \mathcal{T}$. According to our assumption this
contradicts with \cite[Lemma 2.7]{BenkovicEremita1}. Thus
$\lambda\neq 0$.

Now we define a linear mapping $\mathfrak{m}\colon
\mathcal{T}\rightarrow \mathcal{T'}$ by
$$
\mathfrak{m}(x)=\lambda\mathfrak{l}(x)+\frac{1}{2}\mu(x) \eqno(4.3)
$$
for the $x\in \mathcal{T}$. Of course, $\mathfrak{m}$ is a linear
mapping. Our goal is to show that $\mathfrak{m}$ is a Jordan
homomorphism. In view of $(4.2)$ and $(4.3)$, we have
$$
\mathfrak{m}(x^{2})=\lambda\mathfrak{l}(x^{2})+\frac{1}{2}\mu(x)=
\lambda^{2}\mathfrak{l}(x)^{2}+\lambda\mu(x)\mathfrak{l}(x)+\lambda\nu(x)+\frac{1}{2}\mu(x^{2}),
$$
while
$$
\mathfrak{m}(x)^{2}=(\lambda\mathfrak{l}(x)+\frac{1}{2}\mu(x))^{2}
=\lambda^{2}\mathfrak{l}(x)^{2}+\lambda\mu(x)\mathfrak{l}(x)+\frac{1}{4}\mu(x)^{2}.
$$
Comparing the above two identities we get
$$
\mathfrak{m}(x^{2})-\mathfrak{m}(x)^{2}\in \mathcal{Z(T^\prime)}
\eqno(4.4)
$$
for all $x\in \mathcal{T}$. Linearizing $(4.4)$ we obtain
$$
\mathfrak{m}(x\circ y)-\mathfrak{m}(x)\circ \mathfrak{m}(y)\in
\mathcal{Z(T^\prime)}
$$
for all $x, y\in \mathcal{T}$. Define the mapping $\varepsilon:
\mathcal{T}\times \mathcal{T}\rightarrow \mathcal{Z(T')}$ by
$$
\varepsilon(x,y)=\mathfrak{m}(x\circ y)-\mathfrak{m}(x)\circ
\mathfrak{m}(y). \eqno(4.5)
$$
Clearly, $\varepsilon$ is a symmetric bilinear mapping. Of course,
$\mathfrak{m}$ is a Jordan homomorphism if and only if
$\varepsilon(x, y)=0$ for all $x, y\in \mathcal{T}$. For any $x,y\in
\mathcal{T}$, let us put $W=\mathfrak{m}(x\circ (x\circ y))$. By
$(4.5)$ we have
$$
\begin{aligned}
W&=\mathfrak{m}(x)\mathfrak{m}(x\circ y)+\mathfrak{m}(x\circ y)\mathfrak{m}(x)+\varepsilon(x,x\circ y) \\
&=\mathfrak{m}(x)\{\mathfrak{m}(x)\circ \mathfrak{m}(y)+\varepsilon(x,y)\}+[\mathfrak{m}(x)\circ\mathfrak{m}(y)+\varepsilon(x,y)]\mathfrak{m}(x)+\varepsilon(x,x\circ y) \\
&=\mathfrak{m}(x)^{2}\mathfrak{m}(y)+2\mathfrak{m}(x)\mathfrak{m}(y)\mathfrak{m}(x)+\mathfrak{m}(y)\mathfrak{m}(x)^{2}+2\varepsilon(x,y)\mathfrak{m}(x)+\varepsilon(x,x\circ
y).
\end{aligned}
$$
On the other hand
$$
\begin{aligned}
W&=2\mathfrak{m}(xyx)+\mathfrak{m}(x^{2}\circ y) \\
&=2\mathfrak{m}(xyx)+\mathfrak{m}(x^{2})\circ \mathfrak{m}(y)+\varepsilon(x^{2},y) \\
&=2\mathfrak{m}(xyx)+[\mathfrak{m}(x^{2})+\frac{1}{2}\varepsilon(x,x)]\mathfrak{m}(y) \\
&\quad +\mathfrak{m}(y)[\mathfrak{m}(x^{2})+\frac{1}{2}\varepsilon(x,x)]+\varepsilon(x^{2},y) \\
&=2\mathfrak{m}(xyx)+\mathfrak{m}(x)^{2}\mathfrak{m}(y)+\mathfrak{m}(y)\mathfrak{m}(x)^{2} \\
&\quad +\varepsilon(x,x)\mathfrak{m}(y)+\varepsilon(x^{2},y).
\end{aligned}
$$
Comparing the above two relations gives
$$
\begin{aligned}
\mathfrak{m}(xyx)&=\mathfrak{m}(x)\mathfrak{m}(y)\mathfrak{m}(x)+\varepsilon(x,y)\mathfrak{m}(x)-\frac{1}{2}\varepsilon(x,x)\mathfrak{m}(y) \\
&\quad +\frac{1}{2}\varepsilon(x,x\circ
y)\mathfrak{m}(y)-\frac{1}{2}\varepsilon(x^{2},y).
\end{aligned}\eqno(4.6)
$$
By completing linearization of $(4.6)$ we obtain
$$
\begin{aligned}
\mathfrak{m}(xyz+zyx)&=\mathfrak{m}(x)\mathfrak{m}(y)\mathfrak{m}(z)+\mathfrak{m}(z)\mathfrak{m}(y)\mathfrak{m}(x)+\varepsilon(x,y)\mathfrak{m}(z) \\
&\quad +\varepsilon(z,y)\mathfrak{m}(x)-\varepsilon(x,z)\mathfrak{m}(y)+\frac{1}{2}\varepsilon(x,z\circ y) \\
&\quad +\frac{1}{2}\varepsilon(z,x\circ
y)-\frac{1}{2}\varepsilon(x\circ z,y).
\end{aligned}\eqno(4.7)
$$
Let us consider $U=\mathfrak{m}(xyx^{2}+x^{2}yx)$. By $(4.7)$ we
know that
$$
\begin{aligned}
U&=\mathfrak{m}(x)\mathfrak{m}(y)\mathfrak{m}(x^{2})+\mathfrak{m}(x^{2})\mathfrak{m}(y)\mathfrak{m}(x)+\varepsilon(x,y)\mathfrak{m}(x^{2}) \\
&\quad +\varepsilon(x^{2},y)\mathfrak{m}(x)-\varepsilon(x,x^{2})\mathfrak{m}(y)+\frac{1}{2}\varepsilon(x,x^{2}\circ y) \\
&\quad +\frac{1}{2}\varepsilon(x^{2},x\circ
y)-\frac{1}{2}\varepsilon(x^{3},y).
\end{aligned}
$$
Since
$\mathfrak{m}(x^{2})=\mathfrak{m}(x)^{2}+\frac{1}{2}\varepsilon(x,x)$,
we get
$$
\begin{aligned}
U&=\mathfrak{m}(x)\mathfrak{m}(y)\mathfrak{m}(x)^{2}+\mathfrak{m}(x)^{2}\mathfrak{m}(y)\mathfrak{m}(x)+\varepsilon(x,x)\mathfrak{m}(x)\mathfrak{m}(y) \\
&\quad +\frac{1}{2}\varepsilon(x,x)\mathfrak{m}(y)\mathfrak{m}(x)+\varepsilon(x,y)\mathfrak{m}(x)^{2}+\varepsilon(x^{2},y)\mathfrak{m}(x) \\
&\quad -\varepsilon(x,x^{2})\mathfrak{m}(y)+\frac{1}{2}\varepsilon(x,y)\varepsilon(x,x)+\frac{1}{2}\varepsilon(x,x^{2}\circ y) \\
&\quad +\frac{1}{2}\varepsilon(x^{2},x\circ y)-\varepsilon(x^{3},y).
\end{aligned}
$$
On the other hand, using $(4.5)$ and $(4.6)$ we have
$$
\begin{aligned}
U&=\mathfrak{m}((xyx)\circ x) \\
&=\mathfrak{m}(xyx)\circ \mathfrak{m}(x)+\varepsilon(xyx,x) \\
&=\mathfrak{m}(x)\mathfrak{m}(y)\mathfrak{m}(x)^{2}+\mathfrak{m}(x)^{2}\mathfrak{m}(y)\mathfrak{m}(x)+2\varepsilon(x,y)\mathfrak{m}(x)^{2} \\
&\quad -\frac{1}{2}\varepsilon(x,x)(\mathfrak{m}(y)\circ
\mathfrak{m}(x))+\varepsilon(x,x\circ
y)\mathfrak{m}(x)-\varepsilon(x^{2},y)\mathfrak{m}(x)+\varepsilon(xyx,x).
\end{aligned}
$$
Comparing the above two relations yields
$$
\begin{aligned}
&\varepsilon(x,x)\mathfrak{m}(x)\circ \mathfrak{m}(y)-\varepsilon(x,y)\mathfrak{m}(x)^{2}-\varepsilon(x,x^{2})\mathfrak{m}(y) \\
&\quad +(2\varepsilon(x^{2},y)-\varepsilon(x,x\circ
y))\mathfrak{m}(x)\in \mathcal{Z(T^\prime)}
\end{aligned}\eqno(4.8)
$$
for all $x, y\in \mathcal{T}$. In particular, if $x=y$, we obtain
$$
\varepsilon(x,x)\mathfrak{m}(x)^{2}-\varepsilon(x,x^{2})\mathfrak{m}(x)\in
\mathcal{Z(T^\prime)} \eqno(4.9)
$$
for all $x\in \mathcal{T}$. Therefore
$$
\varepsilon(x,x)[[\mathfrak{m}(x)^{2},u],[\mathfrak{m}(x),u]]=0
$$
for all $x\in \mathcal{T}, u\in \mathcal{T'}$, which can be in view
of $(4.3)$ rewritten as
$$
\lambda^{3}\varepsilon(x,x)[[\mathfrak{l}(x)^{2},u],[\mathfrak{l}(x),u]]=0.
$$
We may assume that $A'$ is noncommutative. Pick $a_1,a_2\in A'$ such
that $a_1[a_1,a_2]a_1\neq 0$ (see the proof of \cite[Lemma
2.7]{BenkovicEremita1}). Setting
$$
\mathfrak{l}(x_{0})=\left[\begin{array}{cc}
a_{1} & 0\\
& 0\end{array}\right]\quad \text{and}\quad
u_{0}=\left[\begin{array}{cc}
a_{2} & m\\
& 0\end{array}\right]
$$
for some $x_{0}\in \mathcal{T}$ and an arbitrary $m\in  M'$ in the
relation
$\lambda^{3}\varepsilon(x,x)[[\mathfrak{l}(x)^{2},u],[\mathfrak{l}(x),u]]=0$,
we arrive at
$$
\pi_{A'}(\lambda^{3}\varepsilon(x_{0},x_{0}))a_{1}[a_{1},a_{2}]a_{1}m=0
$$
for all $m\in M^\prime$. By the loyality of $M^\prime$ it follows
that
$\pi_{A'}(\lambda^{3}\varepsilon(x_{0},x_{0}))a_{1}[a_{1},a_{2}]a_{1}=0$.
Hence $\pi_{A'}(\lambda^{3}\varepsilon(x_{0},x_{0}))=0$ by Lemma
\ref{xxsec3.2}. This shows that
$\lambda^{3}\varepsilon(x_{0},x_{0})=0$. Since $\lambda\neq 0$,
$\varepsilon(x_{0},x_{0})=0$ by Lemma \ref{xxsec3.3}. Taking $\varepsilon(x_{0},x_{0})=0$ into $(4.9)$ and then making the commutator with $u_0$ we can obtain $\varepsilon(x_{0},x_{0}^{2})=0$. From $(4.8)$ we get
$$
\varepsilon(x_{0},y)\mathfrak{m}(x_{0})^2
+[-2\varepsilon(x_0^2, y)+\varepsilon(x_{0},x_{0}\circ y)]\mathfrak{m}(x_{0})\in
\mathcal{Z(T^\prime)} \eqno(4.10)
$$
for all $y\in \mathcal{T}$. Let us commute the above relation with $u_0$ and then
with $[\mathfrak{m}(x_{0}), u_0]$ in order. We will eventually observe that $\varepsilon(x_{0},y)=0$ for all $y\in \mathcal{T}$.
Then the equation $(4.10)$ shows that $2\varepsilon(x_0^2, y)\mathfrak{m}(x_{0})\in
\mathcal{Z(T^\prime)}$ for all $y\in \mathcal{T}$. Therefore $\varepsilon(x_0^2, y)[\mathfrak{m}(x_{0}), u_0]=0$
and hence $\varepsilon(x_0^2, y)=0$ for all $y\in \mathcal{T}$.

Next, we assert that $\varepsilon(x,y)=0$. Substituting $x_{0}+y$ for
$x$ by in $(4.9)$ and using the fact $\varepsilon(x_{0},y)=0$, we have
$$
\begin{aligned}
&\varepsilon(y,y)\mathfrak{m}(x_{0})^{2}+\varepsilon(y,y)\mathfrak{m}(x_{0})\circ \mathfrak{m}(y)-\varepsilon(y,(x_{0}+y)^{2})\mathfrak{m}(x_{0}) \\
&-\varepsilon(y,x_{0}\circ y)\mathfrak{m}(y)\in \mathcal{Z(T^\prime)}.
\end{aligned}
$$
On the other hand, replacing $x$ by $-x_{0}+y$ in $(4.9)$ we get
$$
\begin{aligned}
&\varepsilon(y,y)\mathfrak{m}(x_{0})^{2}-\varepsilon(y,y)\mathfrak{m}(x_{0})\circ \mathfrak{m}(y)+\varepsilon(y,(x_{0}-y)^{2})\mathfrak{m}(x_{0}) \\
&+\varepsilon(y,x_{0}\circ y)\mathfrak{m}(y)\in \mathcal{Z(T^\prime)}.
\end{aligned}
$$
Comparing the two relations it follows that
$$
\varepsilon(y,y)\mathfrak{m}(x_{0})^{2}
-\varepsilon(y,x_{0}\circ
y)\mathfrak{m}(x_{0})\in \mathcal{Z(T^\prime)}.
$$
Commuting with $u_{0}$ and then with $[\mathfrak{m}(x_{0}),u_{0}]$, in view of $(4.3)$ the above relation becomes
$$
\varepsilon(y,y)[[\mathfrak{l}(x_{0})^{2},u_{0}],[\mathfrak{l}(x_{0}),u_{0}]]=0.
\eqno(4.11)
$$
Furthermore, $\varepsilon(y,y)=0$ for all $y\in \mathcal{T}$.
Hence $\varepsilon=0$ by the symmetry of $\varepsilon$. This shows that $\mathfrak{m}$ is a Jordan
homomorphism.

We claim that $\lambda=\pm 1$. By $(4.3)$ it follows that
$$
\begin{aligned}
\lambda^{2}\mathfrak{m}([[x,y],z])&=\lambda^{3}\mathfrak{l}([[x,y],z])+\frac{1}{2}\lambda^{2}\mu([[x,y],z]) \\
&=[[\mathfrak{m}(x),\mathfrak{m}(y)],\mathfrak{m}(z)]+\frac{1}{2}\lambda^{2}\mu([[x,y],z])
\end{aligned}
$$
for all $x,y,z\in \mathcal{T}$. Moreover, we get
$$
\lambda^{2}\mathfrak{m}([[x,y],z])-[[\mathfrak{m}(x),\mathfrak{m}(y)],
\mathfrak{m}(z)]\in \mathcal{Z(T^\prime)} \eqno(4.12)
$$
for all $x,y,z\in \mathcal{T}$. Considering $(4.12)$ and using the
facts $\mathfrak{m}(x\circ y)=\mathfrak{m}(x)\circ \mathfrak{m}(y)$ and
$[[x,y],z]=x\circ (y\circ z)-y\circ (x\circ z)$, we conclude that
$$
(\lambda^{2}-1)[[\mathfrak{m}(x),\mathfrak{m}(y)],\mathfrak{m}(z)]\in
\mathcal{Z(T^\prime)}.
$$
for all $x,y,z\in \mathcal{T}$. By $(4.3)$ we know that
$\lambda^{3}(\lambda^{2}-1)\mathfrak{l}([[x,y],z])\in
\mathcal{Z(T^\prime)}$. Since $x,y,z$ are arbitrary elements in
$\mathcal{T}$ and $\mathfrak{l}$ is bijective, we eventually obtain
$\lambda^{3}(\lambda^{2}-1)=0$. Since $\lambda\neq 0$, we get
$\lambda=\pm 1$.

Let us put $\mathfrak{n}(x)=-\frac{1}{2}\mu(x)$. When $\lambda=1$, then $\mathfrak{l}=\mathfrak{m}+\mathfrak{n}$. It is easy to verify that $\mathfrak{n}([[x, y], z])=0$ for all $x, y, z\in \mathcal{T}$. Note that $\mathfrak{m}$ is a Jordan homomorphism from $\mathcal{T}$ into $\mathcal{T}^\prime$ and hence is a Lie triple homomorphism from $\mathcal{T}$ into $\mathcal{T}^\prime$. When $\lambda=-1$, then $\mathfrak{n}=\mathfrak{l}+\mathfrak{m}$ is a Lie triple homomorphism from
$\mathcal{T}$ into $\mathcal{Z(T^\prime)}$. Therefore $\mathfrak{n}([[x, y], z])=0$ for all $x, y, z\in \mathcal{T}$.

We have to prove that $\mathfrak{m}$ is one-to-one. Suppose that
$\mathfrak{m}(w)=0$ for some $w\in \mathcal{T}$. Then
$\mathfrak{l}(w)\in \mathcal{Z(T^\prime)}$ and hence $w\in
\mathcal{Z(T)}$. This implies that ${\rm ker}(\mathfrak{m})\subseteq
\mathcal{Z(T)}$. That is, ${\rm ker}(\mathfrak{m})$ is a Jordan
ideal of $\mathcal{Z(T)}$. However, by Lemma \ref{xxsec4.1} it
follows that ${\rm ker}(\mathfrak{m})=0$.

It remains to prove that $\mathfrak{m}$ is onto in case
$\mathcal{T^\prime}$ is central over $\mathcal{R}$. Let us first
show that $\mathfrak{m}(1)=1^\prime$. Since $\mathfrak{l}$ is a Lie
triple isomorphism, we have $\mathfrak{l}(1)\in
\mathcal{Z(T^\prime)}$ and hence
$\mathfrak{m}(1)=\mathfrak{l}(1)-\mathfrak{n}(1)\in \mathcal{Z(T^\prime)}$.
Note that $\mathfrak{m}$ is a Jordan homomorphism. We see that
$2\mathfrak{m}(x)=\mathfrak{m}(x\circ
1)=2\mathfrak{m}(x)\mathfrak{m}(1)$. Since $\frac{1}{2}\in
\mathcal{R}$, $(\mathfrak{m}(1)-1^{'})\mathfrak{m}(x)=0$, which can
be rewritten as $(\mathfrak{m}(1)-1^{'})\mathfrak{l}(x)\in
\mathcal{Z(T^\prime)}$. Then
$(\mathfrak{m}(1)-1')[\mathfrak{l}(x),s]=0$ for all $s\in
\mathcal{T'}$. Therefore
$(\mathfrak{m}(1)-1')[\mathcal{T'},\mathcal{T'}]=0$. Consequently,
$\pi_{A'}(\mathfrak{m}(1)-1')[A',A']=0$. This implies that
$\pi_{A'}(\mathfrak{m}(1)-1')=0$ and so $\mathfrak{m}(1)=1'$.Obviously, we may write $\mathfrak{n}(x)=f(x)1^\prime$ for some linear mapping $f: \mathcal{T}\longrightarrow \mathcal{R}$. Since $\mathfrak{m}$ is
$\mathcal{R}$-linear, we obtain that
$\mathfrak{l}(x)=\pm\mathfrak{m}(x)+f(x)1'=\mathfrak{m}(\pm
x+f(x)1)$ for all $x\in \mathcal{T}$. Consequently $\mathfrak{m}$
is onto, since $\mathfrak{l}$ is bijective. The proof of the theorem is thus completed.
\end{proof}

It would be helpful to point out that the proof just given in its first part is
a modification of that of \cite[Theorem 2]{Bresar2} and we express it explicitly here for completeness.
By a slight modification of this proof one could easily check the following
proposition holds true.

\begin{proposition}\label{xxsec4.3}
Let $\mathcal{T}$ and $\mathcal{T}^\prime$ be central unital algebras
over a field $F$ with ${\rm char}(F)\neq 2$ and $\mathfrak{l}\colon
\mathcal{T}\longrightarrow \mathcal{T}^\prime$ be a Lie triple isomorphism.
If

\begin{enumerate}
\item[{\rm(1)}] each centralizing trace of a bilinear mapping on $\mathcal{T}^\prime$ is proper,

\item[{\rm(2)}] $\mathcal{T}$ and $\mathcal{T}^\prime$ do not satisfy the polynomial identity $[[x^2, y], [x, y]]$,

\item[{\rm(3)}] $\mathcal{T}^\prime$ does not satisfy the polynomial identity $[x, [y, w]]$,
\end{enumerate}
then $\mathfrak{l}=\pm \mathfrak{m}+\mathfrak{n}$, where
$\mathfrak{m}\colon \mathcal{T}\rightarrow \mathcal{T'}$ is a Jordan
isomorphism and $\mathfrak{n}\colon
\mathcal{T}\longrightarrow \mathcal{Z(T^\prime)}$ is a linear
mapping vanishing on each second commutator.
\end{proposition}

We are in a position to state the main result of this section, which follows from Theorem \ref{xxsec3.4}
and Theorem \ref{xxsec4.2}.

\begin{theorem}\label{xxsec4.4}
Let $\mathcal{T}=\left[\smallmatrix A & M\\
O & B \endsmallmatrix \right]$ and $\mathcal{T^\prime}=\left[\smallmatrix A^\prime & M^\prime\\
O & B^\prime \endsmallmatrix \right]$ be two triangular algebras over $\mathcal{R}$ with
$\frac{1}{2}\in\mathcal{R}$. Let $\mathfrak{l}\colon
\mathcal{T}\longrightarrow\mathcal{T'}$ be a Lie triple isomorphism. If
\begin{enumerate}
\item[{\rm(1)}] each commuting linear mapping on $A'$ or $B'$ is proper,
\item[{\rm(2)}] $\pi_{A'}(\mathcal{Z(T^\prime)})={\mathcal Z}(A')\neq A'$ and
$\pi_{B'}(\mathcal{Z(T^\prime)})={\mathcal Z}(B')\neq B'$,
\item[{\rm(3)}] either $A$ or $B$ is noncommutative,
\item[{\rm(4)}] $M'$ is loyal,
\end{enumerate}
then $\mathfrak{l}=\pm \mathfrak{m}+\mathfrak{n}$, where
$\mathfrak{m}\colon \mathcal{T}\rightarrow \mathcal{T'}$ is a Jordan
homomorphism, $\mathfrak{m}$ is one-to-one, and $\mathfrak{n}\colon
\mathcal{T}\longrightarrow \mathcal{Z(T^\prime)}$ is a linear
mapping vanishing on each second commutator. Moreover, if
$\mathcal{T^\prime}$ is central over $\mathcal{R}$, then
$\mathfrak{m}$ is onto.
\end{theorem}

Beidar, Bre\v{s}ar and Chebotar in \cite{BeidarBresarChebotar1} characterized Jordan isomorphisms of triangular matrix
algebras over a connected commutative ring and obtained the following result. Let $\mathcal{R}$ be a $2$-torsionfree commutative ring with identity $1$ and $\mathcal{T}_n(\mathcal{R})(n\geq 2)$ be the
algebra of all upper triangular $n\times n (n\geq 2)$ matrices over $\mathcal{R}$. Then $\mathcal{R}$ contains no idempotents
except $0$ and $1$ (or equivalently, $\mathcal{R}$ is a connected ring) if and only if every Jordan isomorphism of $\mathcal{T}_n(\mathcal{R})$ onto an arbitrary algebra
over $\mathcal{R}$ is either an isomorphism or an anti-isomorphism. Wong \cite{Wong} extended the previous result by proving that if $\mathcal{T}$ is a $2$-torsion free unital indecomposable triangular algebra, then every Jordan isomorphism from $\mathcal{T}$ onto another algebra is either an isomorphism
or an anti-isomorphism.

\begin{corollary}\label{xxsec4.5}
Let $\mathcal{R}$ be a commutative domain
with $\frac{1}{2}\in \mathcal{R}$ and $\mathcal{T}_n(\mathcal{R})\ (n\geq 2)$ be the algebra of all $n\times n$ upper triangular matrices over $\mathcal{R}$. If ${\mathfrak l}\colon \mathcal{T}_n(\mathcal{R})\longrightarrow \mathcal{T}_n(\mathcal{R})$ is a Lie triple isomorphism, then $\mathfrak{l}=\pm \mathfrak{m}+\mathfrak{n}$, where
$\mathfrak{m}\colon \mathcal{T}_n(\mathcal{R})$ $\longrightarrow \mathcal{T}_n(\mathcal{R})$ is an
isomorphism or an anti-isomorphism and $\mathfrak{n}\colon
\mathcal{T}_n(\mathcal{R})\longrightarrow \mathcal{R}1$ is a linear
mapping vanishing on each second commutator.
\end{corollary}

\begin{proof}
Let us first consider the case of $n=2$. Assume that $\mathfrak{l}\colon \mathcal{T}_2(\mathcal{R}) \longrightarrow \mathcal{T}_2(\mathcal{R})$ is a Lie triple isomorphism.
Denote $E_{ij}$ with $1\leq i\leq j\leq 2$ as the usual matrix unit.
Since $E_{12}=[E_{11}, [E_{11}, E_{12}]]$, we have $\mathfrak{l}(E_{12})=rE_{12}$ for some invertible element $r\in \mathcal{R}^*$. Note that
$[[\mathfrak{l}(I), \mathfrak{l}(X)], \mathfrak{l}(Y)]=0$ for all $X, Y\in \mathcal{T}_2(\mathcal{R})$, which implies that
$[\mathfrak{l}(I), \mathfrak{l}(X)]\in {\mathcal{Z}}(\mathcal{T}_2(\mathcal{R}))=\mathcal{R}I$. Hence
$\mathfrak{l}(I)\in \mathcal{R}I$.
 
We assert that there exists a linear mapping $\mathfrak{g}$ from the diagonal subalgebra $\mathcal{D}_2$ into itself and a scalar $s\in \mathcal{R}$ such that
$$
\mathfrak{l}\left(\left[
\begin{array}
[c]{cc}%
a & b\\
0 & c\\
\end{array}
\right]\right)=\mathfrak{g}\left(\left[
\begin{array}
[c]{cc}%
a & 0\\
0 & c\\
\end{array}
\right]\right)+\left[
\begin{array}
[c]{cc}%
0 & rb+s(a-c)\\
0 & 0\\
\end{array}
\right]
$$
for all $\left[\smallmatrix a & m\\
0 & b \endsmallmatrix \right]\in \mathcal{T}_2(\mathcal{R})$. In fact, we know that for arbitrary $\left[\smallmatrix a & m\\
0 & b \endsmallmatrix \right]\in \mathcal{T}_2(\mathcal{R})$, there exist $\mathcal{R}$-linear mappings $\mathfrak{f}_i, \mathfrak{g}_i, \mathfrak{h}_i\colon \mathcal{R}\longrightarrow \mathcal{R}(i=1,2,3)$ such that
$$
\mathfrak{l}\left(\left[
\begin{array}
[c]{cc}%
a & b\\
0 & c\\
\end{array}
\right]\right)=\left(\left[
\begin{array}
[c]{cc}%
\mathfrak{f}_1(a)+\mathfrak{f}_2(b)+\mathfrak{f}_3(c) & \mathfrak{g}_1(a)+\mathfrak{g}_2(b)+\mathfrak{g}_3(c)\\
0 & \mathfrak{h}_1(a)+\mathfrak{h}_2(b)+\mathfrak{h}_3(c)\\
\end{array}
\right]\right).
$$
Since $\mathfrak{l}(E_{12})=rE_{12}$, we have $\mathfrak{f}_2(b)=0=\mathfrak{h}_2(b)$. Thus
$$
\mathfrak{l}\left(\left[
\begin{array}
[c]{cc}%
a & b\\
0 & c\\
\end{array}
\right]\right)=\left(\left[
\begin{array}
[c]{cc}%
\mathfrak{f}_1(a)+\mathfrak{f}_3(c) & rb+\mathfrak{g}_1(a)+\mathfrak{g}_3(c)\\
0 & \mathfrak{h}_1(a)+\mathfrak{h}_3(c)\\
\end{array}
\right]\right).
$$
Note that $\mathfrak{g}_1(a)=a\mathfrak{g}_1(1)$ and $ \mathfrak{g}_3(c)=c\mathfrak{g}_3(1)$. 
On the other hand, it follows from the fact $\mathfrak{l}(I)\in \mathcal{R}I$ that  $\mathfrak{g}_1(1)+\mathfrak{g}_3(1)=0$. 
Let $\mathfrak{g}_1(1)=s$ and then the above arguments imply our assertion. Let us write $S=\left[\smallmatrix
r & -s\\
0 & 1 \endsmallmatrix \right]$. Then $S^{-1}=\left[\smallmatrix
r^{-1} & sr^{-1}\\
0 & 1 \endsmallmatrix \right]$. 
We define $\mathfrak{j}(T)=\mathfrak{l}(S^{-1}TS)$ for all $T\in \mathcal{T}_2(\mathcal{R})$. Then $\mathfrak{j}$ is a Lie triple isomorphism from $\mathcal{T}_2(\mathcal{R})$ into itself and
$$
\mathfrak{j}\left(\left[
\begin{array}
[c]{cc}%
a & b\\
0 & c\\
\end{array}
\right]\right)=\mathfrak{g}\left(\left[
\begin{array}
[c]{cc}%
a & 0\\
0 & c\\
\end{array}
\right]\right)+\left[
\begin{array}
[c]{cc}%
0 & b\\
0 & 0\\
\end{array}
\right].
$$
This implies that $\mathfrak{j}(E_{12})=E_{12},\ \mathfrak{j}|_{{\mathcal D}_2}=\mathfrak{g}$. 
Note that $\mathfrak{j}$ is obtained by $\mathfrak{l}$ composed with an inner automorphism. Therefore we only to
prove the triple isomorphism $\mathfrak{j}$ is of the standard form.
Suppose that $\mathfrak{j}(E_{11})=\left[\smallmatrix
x & 0\\
0 & y \endsmallmatrix \right]$. By $E_{12}=[\mathfrak{j}(E_{11}), [\mathfrak{j}(E_{11}), E_{12}]]$
it follows that $(x-y)^2=1$. Since $\mathcal{R}$ is a domain, we obtain $x-y=\pm 1$.

{\bf Case 1.} If $x-y=1$, then $\mathfrak{j}(E_{11})=E_{11}+yI$ and $\mathfrak{j}(E_{22})=\mathfrak{j}(I)-\mathfrak{j}(E_{11})=E_{22}+zI$
for some $z\in \mathcal{R}$. It is easy to verify that ${\rm det}(\mathfrak{j})=1+y+z\in \mathcal{R}^*$ as $\mathfrak{j}$ is bijective. In view of \cite[Page 103]{Dokovic}, $\mathfrak{j}$ is of the standard form.

{\bf Case 2.} When $x-y=-1$, note that $-\mathfrak{j}$ is also a triple isomorphism. Define $\mathfrak{t}(T)=\mathfrak{-j}(U^{-1}TU)$ for all $T\in \mathcal{T}_2(\mathcal{R})$, where
$U=\left[\smallmatrix
-1 & 0\\
0 & 1 \endsmallmatrix \right]$. Then
$$
\mathfrak{t}\left(\left[
\begin{array}
[c]{cc}%
a & b\\
0 & c\\
\end{array}
\right]\right)=\mathfrak{g}\left(\left[
\begin{array}
[c]{cc}%
-a & 0\\
0 & -c\\
\end{array}
\right]\right)+\left[
\begin{array}
[c]{cc}%
0 & b\\
0 & 0\\
\end{array}
\right].
$$
This implies that $\mathfrak{t}(E_{12})=E_{12},\ \mathfrak{t}|_{{\mathcal D}_2}=\mathfrak{-g}$.
Moreover, $\mathfrak{t}(E_{11})=\mathfrak{-j}(E_{11})=\left[\smallmatrix
-x & 0\\
0 & -y \endsmallmatrix \right]$. This means that $\mathfrak{t}$ satisfies the assumption of Case 1. Therefore $\mathfrak{t}$
and hence $\mathfrak{j}$ is of the standard form.

Suppose that $n >2$. We may write
$$
\mathcal{T}=\left[
\begin{array}
[c]{cc}%
\mathcal{R} & M_{1\times (n-1)}(\mathcal{R})\\
O & \mathcal{T}_{n-1}(\mathcal{R})\\
\end{array}
\right]
$$
By Corollary \ref{xxsec3.12} each centralizing trace of a bilinear mapping on $\mathcal{T}_n(\mathcal{R})$ is proper. Moreover, $\mathcal{T}_n(\mathcal{R})$ is commutative and $M_{1\times (n-1)}(\mathcal{R})$ is a loyal $(\mathcal{R}, \mathcal{T}_{n-1}(\mathcal{R}))$-bimodule. Thus the assumptions $(1)-(3)$ of Theorem \ref{xxsec4.2} hold in this case. Applying Theorem \ref{xxsec4.2} and \cite[Theorem 3.2]{Wong} yields the conclusion.
\end{proof}

In view of Proposition \ref{xxsec4.3} and \cite[Theorem 3.3]{Wong} we can show

\begin{corollary}\label{xxsec4.6}
Let $\mathcal{N}$ and $\mathcal{N^\prime}$ be nests on a Hilbert space $\textbf{\rm H}$, ${\mathcal Alg}(\mathcal{N})$ and  ${\mathcal Alg}(\mathcal{N^\prime})$ be the nest algebras associated with $\mathcal{N}$
and $\mathcal{N^\prime}$, respectively. If ${\mathfrak l}\colon {\mathcal Alg}(\mathcal{N})\longrightarrow {\mathcal Alg}(\mathcal{N^\prime})$ is a Lie triple isomorphism, then $\mathfrak{l}=\pm\mathfrak{m}+\mathfrak{n}$, where
$\mathfrak{m}\colon {\mathcal Alg}(\mathcal{N})\longrightarrow {\mathcal Alg}(\mathcal{N^\prime})$ is an isomorphism or an anti-isomorphism and $\mathfrak{n}\colon
{\mathcal Alg}(\mathcal{N})\longrightarrow \mathbb{C}1^\prime$ is a linear
mapping vanishing on each second commutator.
\end{corollary}

\begin{proof}
Note that the corollary trivially holds in case ${\rm dim}_{\mathbb C} \textbf{\rm H}=1$ (namely, $\mathfrak{l}={\rm id}+(\mathfrak{l}-{\rm id})$).
If  ${\rm dim}_{\mathbb C} \textbf{\rm H}=2$, we have either ${\mathcal Alg}(\mathcal{N})={\mathcal Alg}(\mathcal{N^\prime})\cong \mathcal{T}_2(\mathbb{C})$ or ${\mathcal Alg}(\mathcal{N})={\mathcal Alg}(\mathcal{N^\prime})\cong \mathcal{M}_2(\mathbb{C})$. Corollary \ref{xxsec4.5}
implies the first case, while the second case follows from \cite[Theorem 3.5]{CalderonGonzalez3}.

Suppose that ${\rm dim}_{\mathbb C} \textbf{\rm H}> 2$. Then each nest algebra is central over $\mathbb{C}$. We assert that the conditions (1)-(3) of Proposition \ref{xxsec4.3} are satisfied in this case. The condition (1) is due to Corollary \ref{xxsec3.13}. While (2) and (3) are due to \cite[Remark 2.13]{BenkovicEremita1}. Applying Proposition \ref{xxsec4.3} and \cite[Theorem 3.3]{Wong} yields the desired result.
\end{proof}

\section{Topics for Further Research}\label{xxsec5}

Although the main purpose of the current article is to study
centralizing traces and Lie triple isomorphisms of triangular
algebras, the structure of centralizing traces and Lie triple
isomorphisms of other associative algebras also has a great interest and draw
more people's our attention. In this section we will present several
potential topics for future further research.

Let us begin with the definition of generalized matrix algebras
given by a Morita context. Let $\mathcal{R}$ be a commutative ring
with identity. A \textit{Morita context} consists of two
$\mathcal{R}$-algebras $A$ and $B$, two bimodules $_AM_B$ and
$_BN_A$, and two bimodule homomorphisms called the pairings
$\Phi_{MN}: M\underset {B}{\otimes} N\longrightarrow A$ and
$\Psi_{NM}: N\underset {A}{\otimes} M\longrightarrow B$ satisfying
the following commutative diagrams:
$$
\xymatrix{ M \underset {B}{\otimes} N \underset{A}{\otimes} M
\ar[rr]^{\hspace{8pt}\Phi_{MN} \otimes I_M} \ar[dd]^{I_M \otimes
\Psi_{NM}} && A
\underset{A}{\otimes} M \ar[dd]^{\cong} \\  &&\\
M \underset{B}{\otimes} B \ar[rr]^{\hspace{10pt}\cong} && M }
\hspace{6pt}{\rm and}\hspace{6pt} \xymatrix{ N \underset
{A}{\otimes} M \underset{B}{\otimes} N
\ar[rr]^{\hspace{8pt}\Psi_{NM}\otimes I_N} \ar[dd]^{I_N\otimes
\Phi_{MN}} && B
\underset{B}{\otimes} N \ar[dd]^{\cong}\\  &&\\
N \underset{A}{\otimes} A \ar[rr]^{\hspace{10pt}\cong} && N
\hspace{2pt}.}
$$
Let us write this Morita context as $(A, B, M, N, \Phi_{MN},
\Psi_{NM})$. We refer the reader to \cite{Morita} for the basic
properties of Morita contexts. If $(A, B, M, N,$ $ \Phi_{MN},
\Psi_{NM})$ is a Morita context, then the set
$$
\left[
\begin{array}
[c]{cc}%
A & M\\
N & B\\
\end{array}
\right]=\left\{ \left[
\begin{array}
[c]{cc}%
a& m\\
n & b\\
\end{array}
\right] \vline a\in A, m\in M, n\in N, b\in B \right\}
$$
form an $\mathcal{R}$-algebra under matrix-like addition and
matrix-like multiplication, where at least one of the two bimodules
$M$ and $N$ is distinct from zero. Such an $\mathcal{R}$-algebra is
usually called a \textit{generalized matrix algebra} of order $2$
and is denoted by
$$
\mathcal{G}=\left[
\begin{array}
[c]{cc}%
A & M\\
N & B\\
\end{array}
\right].
$$
In a similar way, one can define a generalized matrix algebra of
order $n>2$. It was shown that up to isomorphism, arbitrary
generalized matrix algebra of order $n$ $(n\geq 2)$ is a generalized
matrix algebra of order 2 \cite[Example 2.2]{LiWei}. If one of the
modules $M$ and $N$ is zero, then $\mathcal{G}$ exactly degenerates
to an \textit{upper triangular algebra} or a \textit{lower
triangular algebra}. In this case, we denote the resulted upper
triangular algebra (resp. lower triangular algebra) by
$$\mathcal{T^U}=
\left[
\begin{array}
[c]{cc}%
A & M\\
O & B\\
\end{array}
\right]   \hspace{8pt} \left({\rm resp.} \hspace{4pt} \mathcal{T_L}=
\left[
\begin{array}
[c]{cc}%
A & O\\
N & B\\
\end{array}
\right]\right)
$$
Let $\mathcal{M}_n(\mathcal{R})$ be the full matrix algebra
consisting of all $n\times n$ matrices over $\mathcal{R}$. It is
worth to point out that the notion of generalized matrix algebras
efficiently unifies triangular algebras with full matrix algebras
together. The distinguished feature of our systematic work is that
we deal all questions related to (non-)linear mappings of triangular
algebras and full matrix algebras under a unified frame, which is
the admired generalized matrix algebras frame, see \cite{DuWang, LiWei,
LiWykWei, WangWang, XiaoWei1, XiaoWei2}.

Let $\mathcal{T}=\left[\smallmatrix
A & M\\
O & B \endsmallmatrix \right]$ be a $2$-torsion free triangular algebra over commutative
ring $\mathcal{R}$ and ${\mathfrak q}\colon \mathcal{T}\times
\mathcal{T}\longrightarrow \mathcal{T}$ be an $\mathcal{R}$-bilinear
mapping. Theorem \ref{xxsec3.4} shows that under some mild conditions,
every centralizing trace ${\mathfrak T}_{\mathfrak q}:
\mathcal{T}\longrightarrow \mathcal{T}$ of ${\mathfrak q}$ has
the proper form. As you see in the proof of this theorem, one of the most
key steps is that every centralizing trace ${\mathfrak T}_{\mathfrak q}:
\mathcal{T}\longrightarrow \mathcal{T}$ of ${\mathfrak q}$ is commuting.
Bre\v{s}ar in \cite{Bresar1} proved that in certain rings, in
particular, prime rings of characteristic different from $2$ and $3$, every centralizing
trace of arbitrary bilinear mapping is commuting. It is natural to ask the following question

\begin{question}\label{xxsec5.1}
Let $\mathcal{G}=\left[\smallmatrix
A & M\\
N & B \endsmallmatrix \right]$ be a generalized matrix algebra over $\mathcal{R}$ and ${\mathfrak q}\colon \mathcal{G}\times
\mathcal{G}\longrightarrow \mathcal{G}$ be an $\mathcal{R}$-bilinear
mapping. Under what conditions, every centralizing trace ${\mathfrak G}_{\mathfrak q}:
\mathcal{G}\longrightarrow \mathcal{G}$ of ${\mathfrak q}$ has
the proper form ?
\end{question}

Calder\'{o}n Mart\'{i}n and Mart\'{i}n Gonz\'{a}lez in \cite{CalderonGonzalez3} gave a characterization of Lie triple automorphisms of full matrix algebras over complex field $\mathbb{C}$. Let $\mathfrak{l}\colon \mathcal{M}_n(\mathbb{C})\longrightarrow  \mathcal{M}_n(\mathbb{C})(n>1)$ be a Lie triple automorphism. Then there exists an automorphism, an anti-automorphism, the
negative of an automorphism or the negative of an anti-automorphism $\mathfrak{m}\colon \mathcal{M}_n(\mathbb{C})\longrightarrow \mathcal{M}_n(\mathbb{C})$ such that $\mathfrak{n}=\mathfrak{l}-\mathfrak{m}$ is a linear mapping from $\mathcal{M}_n(\mathbb{C})$ onto its
center sending all second commutators to zero. In light of this result and our Theorem \ref{xxsec4.4} we propose

\begin{conjecture}\label{xxsec5.2}
Let $\mathcal{G}=\left[\smallmatrix
A & M\\
N & B \endsmallmatrix \right]$ and $\mathcal{G}^\prime=\left[\smallmatrix
A^\prime & M^\prime\\
N^\prime & B^\prime \endsmallmatrix \right]$ be generalized matrix algebras over $\mathcal{R}$ with
$1/2\in\mathcal{R}$. Let $\mathfrak{l}\colon
\mathcal{G}\longrightarrow\mathcal{G'}$ be a Lie triple isomorphism.
If
\begin{enumerate}
\item[{\rm(1)}] each commuting linear mapping on $A'$ or $B'$ is proper,
\item[{\rm(2)}] $\pi_{A'}(\mathcal{Z(G^\prime)})={\mathcal Z}(A')\neq A'$ and
$\pi_{B'}(\mathcal{Z(G^\prime)})={\mathcal Z}(B')\neq B'$,
\item[{\rm(3)}] either $A$ or $B$ is noncommutative,
\item[{\rm(4)}] $M'$ is loyal,
\end{enumerate}
then $\mathfrak{l}=\pm \mathfrak{m}+\mathfrak{n}$, where
$\mathfrak{m}\colon \mathcal{G}\rightarrow \mathcal{G}^\prime$ is a Jordan
homomorphism, $\mathfrak{m}$ is one-to-one, and $\mathfrak{n}\colon
\mathcal{G}\longrightarrow \mathcal{Z(G^\prime)}$ is a linear
mapping vanishing on each second commutator. Moreover, if
$\mathcal{G^\prime}$ is central over $\mathcal{R}$, then
$\mathfrak{m}$ is surjective.
\end{conjecture}

More recently, some researchers extend the result about Lie isomorphisms between nest algebras on Hilbert spaces by Marcoux and
Sourour \cite{MarcouxSourour2} to the Banach space case, see \cite{QiHou2} and \cite{WangLu}. Therefore it is deserved to pay much more attention to centralizing traces and Lie triple isomorphisms of nest algebras on Banach spaces.

Basing on Corollary \ref{xxsec3.13} we have the following question.

\begin{question}\label{xxsec5.3}
Let $\textbf{\rm X}$ be a Banach space, $\mathcal{N}$ be a nest of $\textbf{\rm X}$ and ${\mathcal Alg}(\mathcal{N})$ be the nest algebra associated with $\mathcal{N}$. Suppose that ${\mathfrak q}\colon {\mathcal Alg}(\mathcal{N})\times
{\mathcal Alg}(\mathcal{N})\longrightarrow {\mathcal Alg}(\mathcal{N})$ is an $\mathcal{R}$-bilinear
mapping. Then every centralizing trace ${\mathfrak T}_{\mathfrak q}:
{\mathcal Alg}(\mathcal{N})\longrightarrow {\mathcal Alg}(\mathcal{N})$ of ${\mathfrak q}$ is
proper.
\end{question}

Furthermore, similiar to Corollary \ref{xxsec4.6} we conjecture

\begin{conjecture}\label{xxsec5.4}
Let $\mathcal{N}$ and $\mathcal{N^\prime}$ be nests on a Banach space $\textbf{\rm X}$, ${\mathcal Alg}(\mathcal{N})$ and  ${\mathcal Alg}(\mathcal{N^\prime})$ be the nest algebras associated with $\mathcal{N}$
and $\mathcal{N^\prime}$, respectively. If ${\mathfrak l}\colon {\mathcal Alg}(\mathcal{N})\longrightarrow {\mathcal Alg}(\mathcal{N^\prime})$ is a Lie triple isomorphism, then $\mathfrak{l}=\pm\mathfrak{m}+\mathfrak{n}$, where
$\mathfrak{m}\colon {\mathcal Alg}(\mathcal{N})\longrightarrow {\mathcal Alg}(\mathcal{N^\prime})$ is an isomorphism or an anti-isomorphism and $\mathfrak{n}\colon
{\mathcal Alg}(\mathcal{N})\longrightarrow \mathbb{C}1^\prime$ is a linear
mapping vanishing on each second commutator.
\end{conjecture}

\bigskip


\end{document}